\documentclass[11pt,leqno]{article}
\usepackage{amssymb,amsmath,amsthm}
\usepackage[alphabetic,lite]{amsrefs}
\usepackage[all,2cell,ps]{xy}
\numberwithin{equation}{section}
\usepackage{graphicx}

\usepackage{pgf,tikz}

\usepackage{enumerate}
\usepackage{mathrsfs}
\usepackage{color}

\newcommand{\nc}{\newcommand}
\nc{\on}{\operatorname}

\usepackage[colorlinks=true, pdfstartview=FitV, linkcolor=blue,%
citecolor=blue, urlcolor=blue]{hyperref}

\newcommand{\symx}{\mathscr{X}}

\newtheorem{theorem}{Theorem}[section]
\newtheorem{proposition}[theorem]{Proposition}

\newtheorem{corollary}[theorem]{Corollary}

\theoremstyle{definition}

\newtheorem{definition}[theorem]{Definition}
\newtheorem{notation}[theorem]{Notation}
\newtheorem{example}[theorem]{Example}

\newtheorem{remark}[theorem]{Remark}

\newcommand{\RR}{\mathrm{R}}

\newcommand{\rb}{\mathrm{b}}

\newcommand{\C}{{\mathbb{C}}}

\newcommand{\R}{{\mathbb{R}}}

\newcommand{\Z}{{\mathbb{Z}}}

\newcommand{\BBS}{{\mathbb S}}

\def\BBV{{\mathbb V}}

\def\phi{{\varphi}}
\def\epsilon{\varepsilon}

\newcommand{\cor}{{\bf k}}

\def\sha{\mathscr{A}}
\def\shb{\mathscr{B}}
\def\shc{\mathscr{C}}
\def\shd{\mathscr{D}}

\def\shf{\mathscr{F}}

\def\shi{\mathscr{I}}

\def\shm{\mathscr{M}}

\def\sho{\mathscr{O}}

\def\shr{\mathscr{R}}
\def\shs{\mathcal{S}}
\def\sht{\mathscr{T}}

\renewcommand{\ker}{\operatorname{Ker}}

\DeclareMathOperator{\im}{Im}

\newcommand{\rmpt}{{\rm pt}}

\newcommand{\into}{\hookrightarrow}

\renewcommand{\to}[1][]{\xrightarrow[]{#1}}

\newcommand{\isoto}[1][]{\xrightarrow[#1]%
{{\raisebox{-.6ex}[0ex][-.6ex]{$\mspace{1mu}\sim\mspace{2mu}$}}}}

\newcommand{\tto}{\rightrightarrows}

\newcommand{\muhom}{\mu hom}
\newcommand{\muHom}[1][]{\mathrm{Hom}^\mu_{\raise1.5ex\hbox to.1em{}#1}}
\newcommand{\Hom}[1][]{\mathrm{Hom}_{\raise1.5ex\hbox to.1em{}#1}}
\newcommand{\RHom}[1][]{\RR\mathrm{Hom}_{\raise1.5ex\hbox to.1em{}#1}}
\newcommand{\Ext}[2][]{\mathrm{Ext}_{\raise1.5ex\hbox to.1em{}#1}^{#2}}
\renewcommand{\hom}[1][]{{\mathscr{H}\mspace{-4mu}om}_{\raise1.5ex\hbox to.1em{}#1}}
\newcommand{\rhom}[1][]{{\RR\mathscr{H}\mspace{-3mu}om}_{\raise1.5ex\hbox to.1em{}#1}}
\newcommand{\rhomc}[1][]
{{\mathscr{H}\mspace{-3mu}om}^*_{\raise1.5ex\hbox to.1em{}#1}}

\newcommand{\ext}[2][]{{\mathscr{E}xt}_{\raise1.5ex\hbox to.1em{}#1}^{#2}}
\newcommand{\Tor}[2][]{\mathrm{Tor}^{\raise1.5ex\hbox to.1em{}#1}_{#2}}
\newcommand{\tens}[1][]{\mathbin{\otimes_{\raise1.5ex\hbox to-.1em{}{#1}}}}
\newcommand{\ltens}[1][]{\mathbin{\overset{\mathrm{L}}\tens}_{#1}}

\newcommand{\lltens}[1][]{{\mathop{\tens}\limits^{\rm L}}_{#1}}

\newcommand{\etens}[1][]{{\mathbin{\boxtimes}}_{#1}}

\newcommand{\Endo}[1][]{\mathrm{End}_{\raise1.5ex\hbox to.1em{}#1}}

\newcommand{\Aut}[1][]{\mathrm{Aut}_{\raise1.5ex\hbox to.1em{}#1}}
\newcommand{\sect}{\Gamma}
\newcommand{\rsect}{\mathrm{R}\Gamma}

\newcommand{\conv}[1][]{\mathop{\circ}\limits_{#1}}

\newcommand{\oim}[1]{{#1}_*}

\newcommand{\roim}[1]{\RR{#1}_*}

\newcommand{\reim}[1]{\RR{#1}_!}
\newcommand{\opb}[1]{#1^{-1}}

\newcommand{\epb}[1]{#1^{!}}

\newcommand{\md[1]}{{\rm Mod}({#1})}

\newcommand{\Db}{{\cal D} b}

\newcommand{\WF}{\ensuremath{\mathrm{WF}}}

\newcommand{\eqdot}{\mathbin{:=}}

\newcommand{\cl}{\colon}
\newcommand{\scbul}{{\,\raise.4ex\hbox{$\scriptscriptstyle\bullet$}\,}}

\newcommand{\tw}[1]{\widetilde{#1}}

\newcommand{\ol}{\overline}

\newcommand{\lp}{{\rm(}}
\newcommand{\rp}{{\rm)}}

\newcommand{\Rc}{{\R\text{-c}}}
\newcommand{\wRc}{{\text{w-}\R\text{-c}}}

\newcommand{\Sol}{{\shs\mspace{-2.5mu}\mathit{ol}}}

\newcommand{\spec}{{\mathrm{spec}}}

\newcommand{\ba}{\begin{array}}
\newcommand{\ea}{\end{array}}

\newcommand{\bnum}{\begin{enumerate}[{\rm(i)}]}
\newcommand{\enum}{\end{enumerate}}
\newcommand{\banum}{\begin{enumerate}[{\rm(a)}]}
\newcommand{\eanum}{\end{enumerate}}

\newcommand{\eq}{\begin{eqnarray}}
\newcommand{\eneq}{\end{eqnarray}}
\newcommand{\eqn}{\begin{eqnarray*}}
\newcommand{\eneqn}{\end{eqnarray*}}

\newcommand{\Proof}{\begin{proof}}
\newcommand{\QED}{\end{proof}}
\newcommand{\Prop}{\begin{proposition}}
\newcommand{\enprop}{\end{proposition}}

\def\rop{{\rm op}}

\def\Op{{\rm Op}}

\DeclareMathOperator{\id}{id}
\DeclareMathOperator{\Ob}{Ob}
\DeclareMathOperator{\supp}{supp}
\DeclareMathOperator{\ori}{or}
\DeclareMathOperator{\chv}{char}

\newcommand{\Der}[1][]{\mathsf{D}^{#1}}
\newcommand{\Derb}{\Der[\mathrm{b}]}

\newcommand{\Derlb}{\Der[\mathrm{lb}]}

\newcommand{\SSi}{\mathrm{SS}}

\DeclareMathOperator{\musupp}{\mu supp}

\newcommand{\RD}{\mathrm{D}}

\newcommand{\Int}{{\rm Int}}
\newcommand{\coh}{{\rm coh}}

\newcommand{\dT}{{\dot{T}}}
\newcommand{\dTM}{{\dT}^*M}

\newcommand{\sindlim}[1][]{\smash{\mathop{\varinjlim}\limits_{#1}}\,}

\newcommand{\Cd}{\mathrm{C}}

\newcommand{\spa}{\vspace{0.5ex}\noindent}

\newcommand{\Simple}{\mathrm{Simple}}

\title{Microlocal analysis and beyond}

\author{Pierre Schapira}

\begin{document}

\maketitle

\begin{abstract}
We shall explain how the idea of ``microlocal analysis'' of the seventies has been reformulated in the framework of sheaf theory in the eighties and then applied to various branches of mathematics, such as linear partial differential equations  or symplectic topology.
\footnote{key words: microlocal analysis, sheaves, microsupport, D-modules, symplectic topology}
\end{abstract}

\tableofcontents

\section*{Introduction}
Mathematics often treat, in its own language, everyday ideas and the concepts that one encounters in this discipline are  frequently familiar. A good illustration of this fact is the dichotomy local/global. These notions appear almost everywhere in mathematics and there is a tool to handle them, this is sheaf theory, a theory elaborated in the fifties (see~\S~\ref{section:catsandshv}). 

But another notion emerged in the  seventies, that of  ``microlocal'' and being local on a manifold $M$ becomes now a global property with respect to the fibers of the cotangent bundle $T^*M\to M$. 

The microlocal point of view first  appeared in Analysis  with Mikio Sato~\cite{Sa70}, soon followed by H\"ormander~\cite{Ho71, Ho83}, who both introduced among others  the notion  of wave front set. The singularities of a  hyperfunction or a distribution on a manifold $M$ are viewed as the projection on $M$ of singularities living in  the cotangent bundle $T^*M$, more precisely in $\sqrt{-1}T^*M$, and the geometry appearing in the cotangent bundle is  in general much easier to analyze (see \S~\ref{section:microlocalan}). 

This  microlocal point of view was then extended to sheaf theory by Masaki Kashiwara and the author in the eighties (see~\cite{KS82, KS85, KS90}) who introduced  the notion of microsupport of sheaves giving rise to microlocal sheaf theory.  
To a sheaf $F$ on a real manifold $M$, one associates its ``microsupport'' $\musupp(F)$\footnote{$\musupp(F)$ was denoted $\SSi(F)$ in loc.\ cit.,  a shortcut for ``singular support''. }
a closed conic subset   of the cotangent bundle $T^*M$ which
describes the codirections of non-propagation of $F$ (see \S~\ref{section:microlocalshv}). 

Microlocal sheaf theory has many applications and we will give a glance at some of them in \S~\ref{section:applications}. First in the study of linear partial differential equations (D-modules and their solutions), which was the original motivation of this theory. Next in other branches of mathematics and in particular in symplectic topology (see in particular~\cite{Ta08} and~\cite{NZ09}). 
The reason why microlocal sheaf theory is closely connected to symplectic topology is that 
 the microsupport of a sheaf is a co-isotropic subset and the category of sheaves, localized on the cotangent bundle, is a homogeneous symplectic  invariant playing  a role similar to that of  the Fukaya category although the techniques in these two fields  are extremely different.

\vspace{3ex}\noindent
Before entering the core of our subject, we shall briefly recall our basic language, categories, homological algebra and sheaves. Then, following an historical point of view, we will recall the birth of algebraic analysis, with Sato's hyperfunctions in 1959-60, and the birth of microlocal analysis, with Sato's microfunctions around 1970. Then, we will describe some aspects of microlocal sheaf theory  (1980-90) and some of its recent applications in symplectic topology.

\section{The prehistory: categories and sheaves}\label{section:catsandshv}
In everyday life, one often speaks of  ``local'' or ``global'' phenomena (e.g., local wars, global warming) which are now common notions. These two concepts also exist  in Mathematics, in which they have a precise meaning.  On a topological space $X$, a property  is {\em locally satisfied}  if there exists an open covering of $X$ on which it is satisfied. But it can happen that a property  is locally satisfied without  being globally satisfied. For example, an equation may be locally solvable without being globally solvable. Or, more sophisticated, a manifold  is always locally orientable, but not always globally orientable, as shown by the example of the M{\"o}bius strip.  And also, a manifold is locally isomorphic (as a topological space) to an open subset of the Euclidian space $\R^n$, but the $2$-dimensional sphere 
$\BBS^2$ is not globally isomorphic to any open subset of  $\R^2$, and this is why in order to recover the earth by planar maps, one needs at least two maps.

There is a wonderful tool which makes a precise link between local and global properties and which plays a prominent role in mathematics, it is sheaf theory. Sheaf theory was created by Jean Leray when he was a war prisoner in the forties. At the beginning it was aimed at algebraic topology, but its scope goes far beyond and this language is used almost everywhere,  in algebraic geometry, representation theory, linear analysis, mathematical physic, etc. It is a basic and universal language in mathematics. Leray's ideas were extremely difficult to follow, but were clarified by Henri Cartan and Jean-Pierre Serre in the fifties who made sheaf theory an essential tool for analytic and algebraic  geometry\footnote{For a short history of sheaf theory, see the historical notes by Christian Houzel in~\cite{KS90}.}.

But, as everyone knows, in mathematics and mathematical physics
 the set theoretical point of view is often  supplanted by the categorical perspective. Category theory was introduced
by Eilenberg-McLane ~\cite{EM45}, more or less at the same time as sheaf theory,
and fantastically developed by Grothendieck, in particular  in his famous Tohoku paper~\cite{Gr57}.
The underlying idea of category theory is 
that mathematical objects only take their full force in
relation with other objects of the same type.
This   is part of a broader intellectual movement of which structuralism of Claude L{\'e}vi-Strauss and linguistics of Noam Chomsky are illustrations.

The link between categories and sheaves is due to Grothendieck.
In this seminal work of 1957,  he interprets a presheaf of sets $F$  as a contravariant functor defined on the category of all open subsets of a topological space $X$ with values in the category of sets, and  a sheaf is a presheaf which satisfies natural glueing properties (see below). 
Later Grothendieck introduced what is now called ``Grothendieck topologies'' by remarking that there is no need of a topological space to develop sheaf theory. The objects of any category may perfectly play the role of the open sets and it remains to define abstractly what the coverings are. But this is another story that we shall not develop here. 

Note that instead of looking at a sheaf as a functor on the category of open sets, one can  also associate a functor $\sect(U;\scbul)$ to each open set $U$ of $X$, from the category of sheaves to that of sets, namely the functor which to a sheaf $F$ associates $F(U)$, its value on $U$. 
When one considers sheaves with values in the category of modules over a given ring, the functor 
$\sect(U;\scbul)$ is left exact but in general not exact. This is precisely the translation of the fact that certain properties are satisfied locally and not globally. 
Then comes the dawn of modern homological algebra, with the introduction of derived functors, and it appears that the cohomology of a sheaf $F$ on $U$ is recovered by the derived functor of $\sect(U;\scbul)$ applied to $F$. These cohomology objects are a kind of measure of the obstruction to pass from local to global.

Let us be a little more precise, referring to~\cite{SGA4, KS06} for an exhaustive treatment.

\subsection{Categories}
Let us briefly recall what a category is. A category $\shc$ is the data of a set of objects, $\Ob(\shc)$, and 
given two objects $X,Y\in\Ob(\shc)$, a set $\Hom[\shc](X,Y)$, called the set of morphisms from $X$ to $Y$, and for any $X,Y,Z\in\Ob(\shc)$ a map $\circ\cl\Hom[\shc](X,Y)\times\Hom[\shc](Y,Z)\to\Hom[\shc](X,Z)$, these data satisfying some axioms which become extremely natural as soon as one thinks as $X,Y,Z$ as being for example sets, or topological spaces, or vector spaces, and $\Hom[\shc](X,Y)$ as being the set of morphisms from $X$ to $Y$, i.e., maps in case of sets, continuous maps in case of topological spaces and linear maps in case of vector spaces. It is then natural to call  the elements of $\Hom[\shc](X,Y)$ the {\em morphisms} from $X$ to $Y$ and to use the notation   $f\cl X\to Y$ for such a morphism. One shall be aware that in general the objects $X,Y,Z$ etc.  are not sets and a fortiori the morphisms are not maps. One calls the map $\circ$ the {\em composition} of morphisms, and one simply asks two things: the composition is associative, $(f\circ g)\circ h=f\circ(g\circ h)$ and for each object $X$ there exists a morphism $\id_X\cl X\to X$ which plays the role of the identity morphism, that is,  $f\circ\id_X=f$ and  $\id_X\circ g=g$ for any $f\cl X\to Y$ and $g\cl Z\to X$. 

Category theory seems at first glance extremely simple and attractive, but there are traps. Indeed, the class of all sets is not a set, as noticed by  Georg Cantor and later by Bertrand Russell whose argument is  a variant of the Greek paradox ``All Cretans are liars". This leads to inextricable problems, unless one uses the concept of universe (or other equivalent notions such as that of unaccessible cardinals) and add an axiom to Set theory, ``any set belongs to a universe'', what Grothendieck did, but perhaps what scared Bourbaki, who never introduced category theory in his globalizing project.

Applying the philosophy of categories to themselves, we have to understand what a morphism 
$F\cl\shc\to\shc'$ from a category $\shc$ to a category $\shc'$ is. Such a morphism is called {\em a functor}. It sends $\Ob(\shc)$ to $\Ob(\shc')$ and  any morphism $f\cl X\to Y$ to a morphism $F(f)\cl F(X)\to F(Y)$. Of course, one asks that $F$ commutes with the composition of morphisms, $F(g\circ f)=F(g)\circ F(f)$, and sends identity morphisms in $\shc$ to identity morphisms in $\shc'$. In practice, one often encounters ``contravariant functors'', that is, functors which reverse the arrows, $F(g\circ f)=F(f)\circ F(g)$. It is better to consider them as usual functors from  $\shc^\rop$ to  $\shc'$, where $\shc^\rop$, the opposite category of $\shc$, is the category $\shc$ in which the arrows are reversed: a morphism $f\cl X\to Y$ in $\shc^\rop$ is a morphism $f\cl Y\to X$ in $\shc$.

\subsection{Homological algebra}\label{subsection:homalg}
Homological algebra is essentially linear algebra, not over a field but over a ring and by extension, in any abelian categories, that is, categories which are modelled on the category of modules over a ring. 

Consider first a finite system of linear equations over a (not necessarily commutative) ring $\cor$:
\eq\label{eq:homal1}
&&\sum_{i=1}^{N_0} a_{ji}u_i=v_j, \quad j=1,\dots,N_1.
\eneq
Here $u_i$ and $v_j$ belong to some left $\cor$-module $S$ and $a_{ji}$ belongs to $\cor$. Denote by $P_0$ the matrix 
$(a_{ji})_{1\leq i\leq N_0, 1\leq j\leq N_1}$ and by $P_0\cdot$ this matrix acting on the left on $S^{N_0}$:
\eqn\label{eq:homal2}
&& S^{N_0}\to[P_0\cdot]S^{N_1}.
\eneqn
Now consider $\cdot P_0$, the matrix $P_0$ acting on the right on $\cor^{N_0}$, and denote by $M$ its cokernel, so that we have {\em an exact sequence}:
\eq\label{eq:homal3}
&& \cor^{N_1}\to[\cdot P_0]\cor^{N_0}\to M\to 0.
\eneq
Conversely, consider a $\cor$-module $M$ and assume that there exists an exact sequence~\eqref{eq:homal3}. 
Then, one says that $M$ admits a finite 1-presentation, but such a presentation  is not unique and different matrices with entries in $\cor$ may give isomorphic modules. 
This is similar to the fact that a finite dimensional vector space may have different systems of generators. As we shall see, when analyzing the system~\eqref{eq:homal1}, the important (and ``intrinsic'') information is not the matrix $P_0$ but the 
module\footnote{According to Mikio Sato (personal communication), at the origin of this idea is the mathematician and philosopher of the 17th century, E.~W von Tschirnhaus.} $M$.

Applying the contravariant  left exact functor $\Hom(\scbul,S)$ to~\eqref{eq:homal3} we find the exact sequence
\eqn\label{eq:homal4}
&& 0\to \Hom(M,S)\to S^{N_0}\to[P_0\cdot]S^{N_1},
\eneqn
which shows  that the kernel of $P_0\cdot$ depends only on $M$ (up to isomorphism) not on the presentation~\eqref{eq:homal3}. 

Assume now that $\cor$ is right Noetherian. Then the kernel of $\cdot P_0$ in~\eqref{eq:homal3}
is finitely generated and one can extend the exact sequence~\eqref{eq:homal3} to an exact sequence
\eq\label{eq:homal5}
&&\cor^{N_2}\to[\cdot P_1] \cor^{N_1}\to[\cdot P_0]\cor^{N_0}\to M\to0.
\eneq
By iterating this construction, one finds a long exact sequence
\eq\label{eq:homal6}
&&\cdots \to \cor^{N_2}\to[\cdot P_1] \cor^{N_1}\to[\cdot P_0]\cor^{N_0}\to M\to0.
\eneq
Consider  the complex
$M^\scbul\eqdot\cdots \to \cor^{N_2}\to[\cdot P_1] \cor^{N_1}\to[\cdot P_0]\cor^{N_0}\to 0$
and identify  $M$ with a complex {\em concentrated in degree $0$}. We get a morphism of complexes $M^\scbul\to M$ 
\eqn
&&\xymatrix{
M^\scbul=\ar[d]&\cdots \ar[r] &\cor^{N_2}\ar[r]^-{\cdot P_1}\ar[d]& \cor^{N_1}\ar[r]^-{\cdot P_0}\ar[d]&\cor^{N_0}\ar[r] \ar[d]&0\\
M=                    &\cdots \ar[r] &0\ar[r]                                                &0\ar[r]                                        &M\ar[r]                  &0
}\eneqn
and this morphism is a {\em qis}, a quasi-isomorphism, that is, induces an isomorphism on the  cohomology. 
Now one proves that, up to ``canonical isomorphism'', the complex 
$\Hom(M^\scbul,S)$ does not depend on the choice of the free resolution $M^\scbul$ and one sets
\eq\label{eq:homal8}
&&\RHom(M,S)=\Hom(M^\scbul,S).
\eneq
This object is represented by the complex (which is no more an exact sequence, but the composition of two consecutive arrows is $0$):
\eq\label{eq:homal9}
&& 0\to S^{N_0}\to[P_0\cdot]S^{N_1}\to[P_1\cdot]\cdots
\eneq
One sets
\eqn
\Ext[]{j}(M,S)&=& H^j\RHom(M,S)\\
&\simeq&\ker(P_j\cl S^{N_j}\to S^{N_{j+1}})/\im(P_{j-1}\cl S^{N_{j-1}}\to S^{N_{j}}).
\eneqn
Hence, $\Ext[]{0}(M,S)\simeq \Hom(M,S)\simeq\ker(P_0)$ and for $j>0$, 
$\Ext[]{j}(M,S)$ is the obstruction for solving  the equation $P_{j-1}u=v$, knowing that $P_{j}v=0$.

One calls $\RHom(\scbul,S)$ the {\em right derived functor} of the left exact (contravariant) functor $\Hom(\scbul,S)$ and this construction is a toy model for the general construction of derived functors.
Indeed, consider a left exact functor  of abelian categories $F\cl \sha\to\shc$. Its right derived functor $RF$ is defined as follows. Given $X$ an object of $\sha$, one first constructs 
(when it is possible)  an  exact sequence
\eq\label{eq:homal10}
&&0\to X\to I^0\to I^1\to\cdots
\eneq
where the $I^j$'s are {\em injective} objects of $\sha$. Let us denote by $I^\scbul$ the complex $0\to I^0\to I^1\to\cdots$ and set $RF(X)=F(I^\scbul)$. (The fact that the arrows in~\eqref{eq:homal5} go backward with respect to~\eqref{eq:homal10} is due to the fact that 
the functor $\Hom(\scbul,S)$ is contravariant.)

The construction of derived functors finds its natural place in the language of derived categories, again due to Grothendieck. The  derived category  $\RD(\shc)$ of an abelian category is obtained by considering complexes in $\shc$ and identifying two complexes when they are quasi-isomorphic. When considering bounded complexes (those whose objects are all $0$ except finitely many of them), one gets the bounded derived category $\Derb(\shc)$. 

Derived categories are particular cases of triangulated categories. In such categories we have ``distinguished triangles'', which play the role of exact sequences in abelian categories.
We shall not say more here. 

\subsection{Abelian sheaves}\label{subsection:abel}
Consider now a topological space $X$ and the set $\Op_X$ of its open sets. This set may be considered as a category by deciding that the morphisms are the inclusions (one morphism if $U\subset V$, no morphisms otherwise). Denote by $\md[\cor]$ the abelian category of left $\cor$-modules. A presheaf of $\cor$-modules is nothing but a functor $F\cl \Op_X^\rop\to\md[\cor]$. Hence, to any open set $U$, $F$ associates a $\cor$-module $F(U)$ and for $U\subset V$ we get a $\cor$-linear map $F(V)\to F(U)$, called the {\em restriction morphism}. Of course, the composition of the restriction morphisms associated with inclusions $U\subset V\subset W$ is the restriction morphism associated with $U\subset W$, and the restriction associated with  $U\subset U$ is the identity. If $s\in F(V)$ and $U\subset V$, one often simply denotes by $s\vert_U$ its image in $F(U)$ by the restriction morphism.

A sheaf $F$ is a presheaf satisfying natural glueing conditions. Namely, for any open subset $U$ of $X$ consider an open covering $U=\bigcup_{i\in I}U_i$. One asks\\
(i) if $s\in F(U)$ satisfies $s\vert_{U_i}=0$ for all $i\in I$, then $s=0$,\\
(ii) if one is given a family $\{s_i\}_{i\in I}$ with $s_i\in F(U_i)$ satisfying $s_i\vert_{U_i\cap U_j}=s_j\vert_{U_i\cap U_j}$ for all pairs $(i,j)$, then there exists $s\in F(U)$ such that $s\vert_{U_i}=s_i$. 
 
 The presheaf $C_X^0$ of $\R$-valued continuous functions on $X$ is a first example of a sheaf (here,  $\cor=\R$), as well as the sheaf $\cor_X$ of locally constant $\cor$-valued functions.  More generally, if $Z$ is a locally closed subset of $X$ (the intersection of an open and a closed subset), there exists a unique sheaf $\cor_{ZX}$ on $X$ whose restriction to $Z$ is the constant sheaf $\cor_Z$ on $Z$ and which is $0$ on $X\setminus Z$. 
 If there is no risk of confusion, we shall simply write $\cor_Z$ instead of $\cor_{ZX}$.
 On the other hand, the presheaf of {\em constant} functions on $X$ is not a sheaf in general.
 
 One easily proves that the category $\md[\cor_X]$ of sheaves on $X$ is an abelian category. One denotes by $\Derb(\cor_X)$ its bounded derived category.

The open set $U$ being given,  consider the functor 
 \eq
 &&\sect(U;\scbul)\cl\md[\cor_X]\to\md[\cor], \quad F\mapsto F(U).
 \eneq
 One easily shows the isomorphism of functors 
 \eqn
 &&\sect(U;\scbul)\simeq\Hom(\cor_U,\scbul).
 \eneqn
 This functor is left exact but not right exact in general. For example, take for $X$ the complex line and consider the complex of sheaves
 \eqn
 &&0\to\C_X\to\sho_X\to[\partial_z]\sho_X\to 0.
 \eneqn
 Here, $\sho_X$ is the sheaf of holomorphic functions and $\partial_z$ the holomorphic derivation. This sequence is exact due to the fact that a holomorphic function  is locally constant if and only if its derivative is $0$ and {\em locally} on $X$, one can solve the equation $\partial_zf=g$. For any non-empty connected open set $U$ the sequence 
 \eq
 &&0\to\C\to\sho_X(U)\to[\partial_z]\sho_X(U)
 \eneq
 remains exact, but  one cannot solve the equation $\partial_zf=g$ globally on $U$ when $U=\C\setminus\{0\}$ and $g(z)=1/z$. Hence, the functor $\sect(U;\scbul)$ is left exact but is not right exact.  Deriving it, one gets the functor $\rsect(U;\scbul)\cl \Derb(\cor_X)\to \Derb(\cor)$ of derived categories.
 
We have chosen to describe the functor $\rsect(U;\scbul)$ but it is a particular case of one of 
six natural functors, called the six Grothendieck operations. One has  the internal hom functor $\rhom$  and the tensor product functor  $\ltens$, the functor of direct  images $\roim{f}$ and its left adjoint $\opb{f}$, the functor of proper direct images $\reim{f}$ and its right adjoint $\epb{f}$. In these Notes, we shall make use of the duality functor $\RD'_X\eqdot\rhom(\scbul,\cor_X)$. 

Given an open set $U$ and setting $S=X\setminus U$, we have an exact sequence of sheaves
\eqn
&&0\to\cor_U\to\cor_X\to \cor_S\to0.
\eneqn
For $F\in\Derb(\cor_X)$, applying the functor  $\RHom(\scbul,F)$ we get the distinguished triangle
\eqn
&& \rsect_S(X;F)\to \rsect(X;F)\to\rsect(U;F)\to[+1].
\eneqn 
Hence, $\rsect_S(X;F)$ is the obstruction to extend uniquely to  $X$ the cohomology classes of $F$ defined on $U$. If $F$ is a usual sheaf, then $H^0\rsect_S(X;F)=\sect_S(X;F)$ is the space of sections of $F$ on $X$ supported by $S$. 

Now consider a sheaf of  rings $\shr$ on $X$. 
Replacing the constant ring $\cor$ with  $\shr$, a system of linear equations on $\shr$ becomes an $\shr$-module which is {\em locally} of finite presentation. If $\shr$ is {\em coherent}, whatever it means, such a module is called a coherent $\shr$-module. Hence, the slogan is ``a system of linear equations with coefficients in $\shr$ is nothing but a coherent left $\shr$-module''. 
We denote by  $\Derb(\shr)$ the derived category of left $\shr$-modules. 
We shall encounter such a situation in Section~\ref{subsection:Dmod}.

\subsection{An application: generalized functions}

In the sixties, people were used to work with various spaces of generalized functions constructed with the tools of functional analysis. Sato's construction of hyperfunctions in~\cite{Sa60} is at the opposite of this practice: he uses 
purely algebraic tools and complex analysis. The importance of Sato's definition is twofold:
first, it is purely algebraic (starting with the sheaf of holomorphic functions), and second it highlights  the link between real and complex geometry. Note that the sheaf $\shb_M$ of hyperfunctions on a real analytic manifold $M$ naturally contains the sheaf $\Db_M$ of distributions and has the nice property of being flabby (the restriction morphisms are surjective). 
We refer to~\cites{An07, Sc07} for an exposition of Sato's work.

Consider first the case where
$M$ is an open interval of the real line $\R$ and let $X$ be an open neighborhood 
of $M$ in the complex line $\C$ satisfying $X\cap\R=M$. Denote by $\sho(U)$ the space of holomorphic functions on an open set  $U\subset X$.
The space $\shb(M)$ of hyperfunctions on $M$ is given by
\eq\label{eq;hyperf1}
&&\shb(M)=\sho(X\setminus M)/\sho(X).
\eneq
In other words, a hyperfunction on $M$ is a holomorphic function on $X\setminus M$, and such a function is considered as $0$ if it extends to the whole of $X$. 

It is easily proved, using the solution of the Cousin problem,  that this
space depends only on $M$, not on the choice of $X$, and that the 
correspondence 
$I\mapsto \shb(I)$ ($I$ open in $M$)  defines a  flabby sheaf $\shb_M$ on $M$. 

With Sato's definition, the boundary values always exist and are no more a
limit in any classical sense. 

Sato's definition is motivated by the  well-known fact
that the Dirac function at $0$ is ``the boundary value'' of $\frac{1}{2i\pi}1/z$.
Indeed, if $\phi$ is a $C^0$-function on $\R$, one has
\eqn
&&\phi(0)=\lim_{\epsilon\to[>]0}\frac{1}{2i\pi}\int_\R(\frac{\phi(x)}{x-i\epsilon}-\frac{\phi(x)}{x+i\epsilon})dx.
\eneqn
and one can write formally:
\eqn
&&\delta(x)=\frac{1}{2i\pi}(\frac{1}{x-i0}-\frac{1}{x+i0}).
\eneqn
It follows that for any distribution $u$ on $\R$ with compact support, say $K\subset \R$, $u$ is the boundary value of the  function 
$u\star\frac{1}{2i\pi}1/z$ holomorphic on $\C\setminus K$ and in fact any distribution on $\R$ is the boundary value of a holomorphic function on $\C\setminus\R$. However, there exist holomorphic functions  on $\C\setminus\{0\}$ 
which have no boundary values as a distribution, such as  
the holomorphic function $\exp(1/z)$ defined on $\C\setminus\{0\}$.

In order to extend to the higher dimensional case  the definition of hyperfunctions, a natural idea would be as follows. Denote 
by $M$ a real analytic manifold and by $X$ a complexification of $M$. Then, it would be tempting to define 
$\shb(M)$ by formula~\eqref{eq;hyperf1}. Unfortunately, this does not work since Hartog's theorem tells us that this space is $0$ as soon as $\dim M>1$. (Here and in the sequel, $\dim$ denotes the dimension of a {\em real} manifold.)

Another way is to use local cohomology, that is, the derived functor of the functor $\sect_M(\scbul)$. Indeed, in dimension $1$ one has
\eqn
&&\shb(M)=\sho(X\setminus M)/\sho(X)\simeq H^1_M(X;\sho_X),
\eneqn
where $\sho_X$ is the sheaf of holomorphic functions on $X$,
and the presheaf $H^1_M(\sho_X)$ is a sheaf on $M$. This is the sheaf $\shb_M$. 

Note that Sato invented local cohomology independently from Grothen\-dieck in the 1960s in order to define hyperfunctions. 
On a  real analytic manifold $M$ of dimension $n$, the sheaf $\shb_M$ was originally defined as
\eqn
&&\shb_M=H^n_M(\sho_X)\tens \ori_{M},
\eneqn
after having proved that the groups $H^j_M(\sho_X)$ are $0$ for $j\neq n$.
Here, $\ori_M$ is the orientation sheaf on $M$. 
Since $X$ is oriented, Poincar{\'e}'s duality gives the isomorphism 
$\RD'_X(\C_{M})\simeq \ori_{M}\,[-n]$ 
where  $\RD_X'$ is the duality functor for sheaves on $X$. An equivalent definition of 
the sheaf of hyperfunctions is thus given  by
\eq\label{eq:hyp2}
\shb_M=\rhom(\RD'_X\C_{M},\sho_X).
\eneq
The sheaf $\sha_M$ of real analytic functions is given by 
\eqn
&&\sha_M\eqdot\C_M\tens\sho_X.
\eneqn
Since $\C_M\simeq \RD'_X\RD'_X\C_{M}$,  we get the natural morphism from real analytic functions to hyperfunctions:
\eqn
&&\sha_M\simeq \rhom(\RD'_X\C_{M},\C_X)\tens\sho_X\to \rhom(\RD'_X\C_{M},\sho_X)\simeq \shb_M.
\eneqn

Formula~\eqref{eq:hyp2} opens the door to a vast generalization of distributions and hyperfunctions: one may consider the sheaves (in the derived sense) $\rhom(F,\sho_X)$ where now $F$ is any sheaf on  $X$. This is particularly interesting when $F$ is $\R$-constructible (see Definition~\ref{def:construct} below). 

Similarly as in dimension $1$, one can represent the sheaf $\shb_M$ by using \v{C}ech cohomology of coverings of $X\setminus M$. For example, let $X$ be a Stein open subset of $\C^n$ and set $M=\R^n\cap X$.
Denote by $x$ the coordinates on $\R^n$ and by $x+\sqrt{-1}y$ the coordinates on $\C^n$.
One can recover $\C^n\setminus\R^n$ by $n+1$ open half-spaces
\eqn
&&V_i=\{(x+\sqrt{-1}y;\langle y,\xi_i\rangle>0\},\quad \xi\in\R^n\setminus\{0\}, \quad i=1,\dots,n+1 . 
\eneqn
For $J\subset\{1,\dots,n+1\}$ set $V_J=\bigcap_{j\in J}V_j$. 
Assuming $n>1$, we have the isomorphism $H^n_M(X;\sho_X)\simeq H^{n-1}(X\setminus M;\sho_X)$. 
Therefore, setting $U_J=V_J\cap X$, we have 
\eqn
&&\shb(M)\simeq \sum_{\vert J\vert=n}\sho_X(U_J)/\sum_{\vert K\vert=n-1}\sho_X(U_K).
\eneqn
Then comes naturally the following problem: how to recognize the directions associated with these $U_J$'s?
The answer is given by Sato's microlocalization functor that we shall describe in \S~\ref{section:microlocalan}.

\section{Microlocal analysis}\label{section:microlocalan}

With any real manifold $M$ (say of dimension $n$) are naturally associated two important vector bundles, the tangent bundle $\tau\cl TM\to M$ and its dual, the cotangent bundle $\pi\cl T^*M\to M$. Classically, one interprets a vector $v\in T_{x_0}M$ as the speed at the point $x_0$ of something moving on $M$ and passing at $x_0$. Up to the zero-section and to the action of $\R^+$ on these vector bundles, one may  think of $T_{x_0}M$  as the space of all light rays issued from $x_0$ and of  $T^*_{x_0}M$ as the space of all half-spaces, or walls, passing through $x_0$. 
 
The tangent bundle is more intuitive, but it appears that the cotangent bundle is much more important. It is the {\em phase space}  of the physicists and it is endowed with a fundamental structure, it is a symplectic manifold. Symplectic geometry is a very classical subject  whose origin perhaps goes back to William Hamilton in the first half  of the $19\mathrm{th}$ century. Note that the duality tangent/cotangent reflects the duality observer/observed.

Analysts have known for long, after Petrowski, Hadamard and Leray,  that, given a linear differential operator $P$ on a manifold $X$, its principal symbol $\sigma(P)$ is a well-defined function on the cotangent bundle and  the geometry of 
its characteristic variety (the zeroes of  $\sigma(P)$)  plays a role in the behaviour of its solutions. But the story  of microlocal analysis really  started in the 70's with Mikio Sato.

As already mentioned, a hyperfunction may be represented (not uniquely) as a 
 sum of boundary values of holomorphic functions defined on 
 tuboids\footnote{A tuboid is an open subset of $X$ which, in a local chart at $x_0\in M$,   contains $(\R^n\times\sqrt{-1}\Gamma)\cap W$ for an open  non-empty convex cone $\Gamma$ of $\R^n$ and an open neighborhood $W$ of $x_0$ in $X$.}
 in $X$ and an important problem is to understand from where  these boundary values come. For that purpose, Sato defined the sheaf of microfunctions $\shc_M$ on the conormal bundle $T^*_MX$ to $M$ in $X$ (the dual of the normal bundle, see below)
whose direct image on $M$ is the sheaf $\shb_M$. The reason of the success of this approach is that the study of partial differential equations is much easier  in $T^*_MX$. Consider for example the wave operator $\Box=\partial_t^2-\sum_{i=1}^n\partial_{x_i}^2$ on $\R_t\times\R^n_x$. Its characteristic variety is  a smooth manifold in $\dT^*_MX$ (the bundle $T^*_MX$ with the zero-section removed) and, locally on $\dT^*_MX$, the equation $\Box u=0$ can be reduced after a ``quantized contact transformation'' to the equation $\partial_tu=0$. 

The breakthrough of microlocal analysis quickly spread from the analytic framework to the $C^\infty$-framework, under the impulse of Lars H\"ormander  who replaced the use of holomorphic functions by that of the Fourier transform. Note that the $\sqrt{-1}$ which appears in the Fourier transform is related (in a precise sense, via the Laplace transform, see~\cite{KS97}) to the isomorphism of vector bundles $T^*_MX\simeq\sqrt{-1}T^*M$, where $T^*M$ is the cotangent bundle to $M$. 

Microfunctions are certainly an important tool of analysis, but in our opinion their construction is still more important. 
Indeed, they are constructed by applying a new functor to the sheaf $\sho_X$,  the functor $\mu_M$ of microlocalization along $M$, and this functor is obtained as the ``Fourier--Sato transform'' of the specialisation functor $\nu_M$. We shall now describe these three functors which are defined in a purely real setting. 
References are made to~\cite{KS90}.

\subsection{Specialization}
\begin{notation}
Let $\BBV\to M$ be a real vector bundle. We identify $M$ with the zero-section of $\BBV$. 
For $Z\subset \BBV$ we denote by $Z^a$ its image by the antipodal map, $(x;v)\mapsto(x;-v)$. We say that a subset $Z$ of $\BBV$ is $\R^+$-conic if it is invariant by the action of $\R^+$ on $\BBV$. 
\end{notation}
Let $X$ be a {\em real} manifold and $\iota\cl M\hookrightarrow X$ the  embedding of a closed submanifold $M$. Denote by 
$\tau_M\cl T_MX\to M$ the normal bundle to $M$ in $X$ defined by the exact sequence of vector bundles over $M$:
\eqn
&&0\to TM\to M\times_XTX\to T_MX\to0.
\eneqn
If $F$ is a sheaf on $X$, its restriction to $M$, denoted $F\vert_M$, may be viewed as a global object, namely the 
direct image by $\tau_M$ of a conic sheaf $\nu_MF$ on $T_MX$, called the specialization of $F$ along $M$. Conic means locally constant on the orbits of the action of $\R^+$. 
Intuitively, $T_MX/\R^+$ is the set of light rays issued from $M$ in $X$ and 
the germ of $\nu_MF$ at a normal vector $(x;v)\in T_MX$ is the germ at $x$ of the restriction of $F$ along the light ray $v$.

In order to construct the specialization, one first constructs a commutative  diagram of manifolds, called the normal deformation of $X$ along 
$M$:
\eqn
&&\xymatrix{
T_MX\ar@{^(->}[r]^-s\ar[d]_-{\tau_M}&\tw X_M\ar[d]^-p&\Omega\ar@{_(->}[l]_-j\ar[ld]^-{\tw p}\\
M\ar@{^(->}[r]_-{\iota}&X
},
\left.\begin{array}{ccccc}
&\\
&\\
&t\cl \tw X_M\to\R,\\
&\Omega=\opb{t}(\R_{>0}),\\
&T_MX\simeq \opb{t}(0),\\
\end{array}\right.
\eneqn
Locally, after choosing a local coordinate system $(x',x'')$ on $X$ such that $M=\{x'=0\}$, we have 
$\tw X_M=X\times\R$, $t\cl \tw X_M\to\R$ is the projection, 
$p(x',x'',t)=(tx',x'')$.

The specialization allows one to define intrinsically the notion of Whitney normal cone.
Let $S\subset X$ be a locally closed subset. The Whitney normal cone $C_M(S)$ is the closed conic subset of $T_MX$ given by
\eq\label{eq:WnormalC1}
&&C_M(S)=\ol{\opb{\tw p}(S)}\cap T_MX.
\eneq
One also defines the normal cone for two subsets $S_1$ and $S_2$ by using the diagonal $\Delta$ of $X\times X$ and  setting
\eq\label{eq:WnormalC2}
&&C(S_1,S_2)=C_\Delta(S_1\times S_2).
\eneq

The Whitney normal cone $C_M(S)$ is given in a local coordinate system $(x)=(x',x'')$ on $X$ with $M=\{x'=0\}$ by
\eqn
&&\left\{
\parbox{65ex}{
$(x''_0;v_0)\in C_M(S)\subset T_MX$ if and only
if there exists a sequence $\{(x_n,c_n)\}_n\subset S\times\R^+$ with $x_n=(x'_n,x''_n)$ such
that $x'_n\to[n]0$, $x''_n\to[n] x''_0$ and $c_n(x'_n)\to[n]v_0$.
}
\right.
\eneqn
\begin{example}
Assume that $M=\{x_0\}$. Then $C_{\{x_0\}}(S)$ is the tangent cone to $S$. 

\spa
(i) If $X=\C^n$, $x_0=0$ and 
$S=\{x\in X;f(x)=0\}$ for a holomorphic function $f$, then $C_{\{0\}}(S)=\{x\in X;g(x)=0\}$ where $g$ is the homogeneous polynomial of lowest degree in the Taylor expansion of $f$ at $0$.

\spa
(ii) Let $\BBV$ be a real finite dimensional vector space and let $\gamma$ be a closed cone.
Then $C_{0}(\gamma)=\gamma$ and $C_{0}(\gamma,\gamma)$ is the vector space generated by $\gamma$.
\end{example}

The specialization functor 
\eqn
&&\nu_M\cl\Derb(\cor_X)\to\Derb(\cor_{T_MX})
\eneqn
is then given by a  formula mimicking~\eqref{eq:WnormalC1}:
\eqn
&&\nu_MF\eqdot \opb{s}\roim{j}\opb{\tw p}F.
\eneqn
Clearly, $\nu_MF\in\Derb_{\R^+}(\cor_{T_MX})$, that is, $\nu_MF$  is an $\R^+$-conic sheaf. Moreover,
\eqn
&&\roim{\tau_M}\nu_MF\simeq \nu_MF\vert_M\simeq F\vert_M.
\eneqn
For an open cone $V\subset T_MX$, one finds that 
\eq\label{eq:sectspecial}
&&H^j(V;\nu_MF)\simeq\sindlim[U]H^j(U;F)
\eneq
where $U$ ranges through the family of open subsets of $X$ such that 
$C_M(X\setminus U)\cap V=\varnothing$. 
\begin{center}
{\begin{tikzpicture}
\draw (0,0) -- (1.5,1.5);
\draw (0,0) -- (1.5,-1.5);
\draw [dashed] (0,0) ..controls +(2,2) and  +(2,-2)  .. (0,0);
\draw (-0.3,0) node {$M$};
\draw (1.8,0.7) node {$V$};
\draw (1.0,0) node {$U$};
\end{tikzpicture}}
\end{center}

In other words, a section of $\nu_MF$ on a conic open set $V$ of $T_MX$ is given by a section of $F$ on a small open set $U$ of $X$ which is, in some sense, tangent to $V$ near $M$. 

\subsection{Fourier--Sato transform}
The classical Fourier transform is an isomorphism 
between a space of (generalized) functions on a real vector space $\BBV$ and another space on the dual space $\BBV^*$. It is an integral transform associated with a kernel on $\BBV\times\BBV^*$. The Fourier--Sato transform is again an integral transform but now in the language of the six Grothendieck operations for  sheaves.
It induces an equivalence of categories between conic sheaves on a vector bundle and conic sheaves on the dual vector bundle. It seems to have been the first integral transform for sheaves and, as its name may suggest, this construction is due to Mikio Sato.

Consider a diagram of  real vector bundles over a real manifold $M$ 
\eqn
&&\xymatrix@R=2.5ex{ 
&\BBV\times_M\BBV^*\ar[dl]_-{p_1}\ar[dr]^-{p_2}&\\
\BBV\ar[dr]_-\tau&&\BBV^*\ar[dl]^-{\pi}\\
&M&
}\eneqn
If $\gamma$ is a cone in $\BBV$, its polar cone, or dual cone, is given by
\eqn
&&\gamma^\circ=\{(x;\xi)\in \BBV^*;\langle\xi,v\rangle\geq0\mbox{ for all }v\in\gamma_x\}.
\eneqn
Then $\gamma^\circ$ is a closed convex cone of $\BBV^*$, and there is no hope to recover $\gamma$ from $\gamma^\circ$ if $\gamma$ is not convex. 
Things are different with sheaves since we can replace the usual functors on sheaves with their derived version. 
Define:
\eqn
&& P=\{(x,y)\in \BBV\times_M \BBV^*; \langle x,y\rangle \leq 0\}.
\eneqn
Denote by  $\Derb_{\R^+}(\cor_\BBV)$ the subcategory of  $\Derb(\cor_\BBV)$ consisting of conic sheaves. 
The  Fourier--Sato transform of a conic sheaf $F$  is a sheaf theoretical version of the construction of the polar cone (see Example~\ref{exa:fourier} below). It is given by the formula
\eqn
F^\wedge &=& \reim{p_2}(\opb{p_1}F)_{P}
\eneqn
\begin{theorem}\label{th:fouriersato}
The functor ${}^\wedge$ induces an equivalence of categories
\eqn
&&{}^\wedge\cl \Derb_{\R^+}(\cor_\BBV)\isoto\Derb_{\R^+}(\cor_{\BBV^*}).
\eneqn
\end{theorem}
\begin{example}\label{exa:fourier} 
Assume for short that $M=\rmpt$ and let $n=\dim \BBV$. 

\spa
(i) Let $\gamma$ be a closed proper\footnote{A cone is proper if it contains no line.} convex cone in $\BBV$. Then:
\eqn
&&(\cor_\gamma)^\wedge \simeq \cor_{\Int (\gamma^\circ)}.
\eneqn
Here $\Int \gamma^\circ$ denotes the interior of $\gamma^\circ$. 

\spa
(ii) Let $\gamma$ be an open convex cone in $\BBV$. Then: 
\eqn
&&(\cor_\gamma)^\wedge \simeq \cor_{\gamma^{\circ a}}\, [-n].
\eneqn

\spa
(iii) Let $(x)=(x',x'')$ be coordinates on $\R^n$ with $(x')=(x_1,\dots,x_p)$ and $(x'')=(x_{p+1},\dots,x_n)$.
Denote by $(y)=(y',y'')$ the dual coordinates on $(\R^n)^*$. 
Set 
\eqn
&&\gamma=\{x;x'^2-x''^2\geq0\},\quad \lambda=\{y;y'^2-y''^2\leq0\}.
\eneqn
Then $(\cor_\gamma)^\wedge\simeq \cor_\lambda[-p]$. (See~\cite{KS97}.)
\end{example}

\subsection{Microlocalization}
Denote by $\pi_M\cl T^*_MX\to M$ the  conormal bundle to $M$, the dual bundle to $T_MX$, given by the exact sequence of  vector bundles over $M$:
\eqn
&&0\to T^*_MX\to M\times_XT^*X\to T^*M\to0.
\eneqn
The microlocalization of $F$ along $M$, denoted $\mu_MF$, is the Fourier--Sato transform of 
$\nu_MF$, hence is an object of $\Derb_{\R^+}(\cor_{T^*_MX})$. It satisfies:
\eqn
&&\roim{\pi_M}\mu_MF\simeq \mu_MF\vert_M\simeq \rsect_MF.
\eneqn
Roughly speaking, the sections of $\mu_MF$ on an open convex cone $V$ of $T^*_MX$ are the sections of $\nu_MF$ supported by the polar cone to $V$ in $T_MX$. More precisely, by using Theorem~\ref{th:fouriersato} and~\eqref{eq:sectspecial}, we get
\eq\label{eq:sectspecial}
&&H^j(V;\mu_MF)\simeq\sindlim[U,Z]H_{Z\cap U}^j(U;F)
\eneq
where $U$ ranges through the family of open subsets of $X$ such that $U\cap M=\pi(V)$ and $Z$ ranges over the family of closed subsets of $X$ such that $C_M(Z)\subset V^\circ$. 
\begin{center}
\begin{tikzpicture}
\draw (0,0) ..controls +(1,1) and  +(-0.2,0.2)  .. (1,1.5);
\draw (0,0) ..controls +(1,-1) and  +(-0.2,-0.2)  .. (1,-1.5);
\draw [dashed] (1,-1.5) ..controls +(1,1) and  +(0.2,-0.2)  .. (1,1.5);
\draw (0,0) -- (1.5,1.5);
\draw (0,0) -- (1.5,-1.5);
\draw (-0.3,0) node {$M$};
\draw (1.8,0.7) node {$V^\circ$};
\draw (0.8,0) node {$U\cap Z$};
\end{tikzpicture}
\end{center}

\label{fig:microlocalization}

\subsection{Application: microfunctions and wave front sets}\label{subsection:microf}
Assume now that $M$ is a real analytic manifold of dimension $n$ and $X$ is a complexification of $M$. 
First notice the isomorphisms 
\eqn
&&M\times_XT^*X\simeq \C\tens_\R T^*M\simeq T^*M\oplus\sqrt{-1}T^*M.
\eneqn
One deduces the  isomorphism 
\eq\label{eq:sqrt1}
&&T^*_MX\simeq \sqrt{-1}T^*M. 
\eneq
Mikio Sato introduced in~\cite{Sa70} the sheaf 
$\shc_M$ of microfunctions on $T^*_MX$ as
\eq\label{eq:hyp}
&&\shc_M=H^n(\mu_M(\sho_X))\tens\opb{\pi}\ori_M,
\eneq
after having proved that the other cohomology groups are $0$. Thus $\shc_M$ is a conic sheaf on $T^*_MX$ and one has a natural isomorphism
\eqn
&&\shb_M\isoto\oim{\pi_M}\shc_M.
\eneqn
Denote by $\spec$ the natural map:
\eqn
&&\spec\cl \sect(M;\shb_M)\isoto\sect(T^*_MX;\shc_M).
\eneqn
\begin{definition}
The wave front set $\WF(u)$ of a hyperfunction $u\in\shb(M)$ is the support of $\spec(u)$.
\end{definition}
\begin{example}
Denote by $(z_1,z_2)$ the coordinates on $X=\C^2$, with $z_j=x_j+\sqrt{-1}y_j$. Let $M=\R^2$ and let $(x_1,x_2;\sqrt{-1}\eta_1,\sqrt{-1}\eta_2)$ denote the coordinates on $T^*_MX$. 
The function $(z_1+\sqrt{-1}z_2^2)^{-1}$ defines a holomorphic function $f$ in the tuboid $\R^2\times\sqrt{-1}\{y_1>y_2^2\}$. 
The boundary value of $f$ on $\R^2$ is a hyperfunction $u$, real analytic for $(x_1,x_2)\neq(0,0)$ and whose wave front set above 
$(0,0)$ is the half-line $\{\eta_1\geq0,\eta_2=0\}$.

Note that $(x_2\partial_1+\frac{\sqrt{-1}}{2} \partial_2)u=0$ and this is in accordance with the result of Proposition~\ref{pro:WF} below. Also note that $\WF(u)$ is not co-isotropic (see Definition~\ref{def:coisotropic}) after identifying $T^*_MX$ with 
$\sqrt{-1}T^*M$ and $\sqrt{-1}T^*M$ with the symplectic manifold $T^*M$. 
\end{example}

\begin{remark}
Since the sheaf $\shb_M$ contains the sheaf $\Db_M$ of distributions, one obtains what is called {\em the analytic wave front set} of distributions. 
\end{remark}
Consider a closed convex proper cone $Z\subset T^*_MX$ which contains the zero-section $M$.
Then,  $\WF(u)\subset Z$ if and only if $u$ is the boundary value $\rb(f)$ of a holomorphic function $f$
 defined in a ``tuboid'' $U$ with ``profile'' the interior of the polar tube to $Z^a$, that is, satisfying
 \eqn
 &&C_M(X\setminus U)\cap \Int Z^{\circ a}=\varnothing.
 \eneqn
Moreover, the sheaf $\shc_M$ is conically flabby. Therefore, any hyperfunction may be decomposed as a sum of boundary values of holomorphic functions $f_i$'s defined in suitable tuboids $U_i$ and 
if we have hyperfunctions 
$u_i$ ($i=1,\dots N$) satisfying $\sum_ju_j=0$, there exist hyperfunctions $u_{ij}$ ($i,j=1,\dots N$) 
such that 
\eqn
&&u_{ij}=-u_{ji}, \quad u_i=\sum_{j=1}^Nu_{ij} \mbox{ and }\WF(u_{ij})\subset\WF(u_i)\cap\WF(u_j).
\eneqn
In other words, 
consider  holomorphic functions $g_i$'s defined in  tuboids $U_i$ and assume that $\sum_i \rb(g_i)=0$. Then  there exist 
 holomorphic functions $g_{ij}$'s defined in  tuboids $U_{ij}$ whose profile is the convex hull of $U_i\cup U_j$ such that
\eqn
&&g_{ij}=-g_{ji}, \quad g_i=\sum_{j=1}^Ng_{ij}.
\eneqn
This is the so-called ``Edge of the wedge theorem'' which was intensively studied in the seventies (see~\cite{Mr67}).

Soon after Mikio Sato had defined the sheaf $\shc_M$ and the analytic wave front set of hyperfunctions, Lars H\"ormander defined the $C^\infty$-wave front set of distributions, by using the classical Fourier transform. See~\cite{Ho71, Ho83} and also~\cite{BI73, Sj82} for related constructions.

\section{Microlocal sheaf theory}\label{section:microlocalshv}

The idea of microlocal analysis was extended to sheaf theory by Masaki Kashiwara and the author (see~\cite{KS82, KS85, KS90}), giving rise to microlocal sheaf theory.  With a sheaf $F$ on a real manifold $M$, one associates its ``microsupport'' $\musupp(F)$\footnote{Concerning the notation $\musupp$, see the footnote in the Introduction.} 
a closed conic subset   of the cotangent bundle $T^*M$. The microsupport 
describes the codirections of non-propagation of $F$. Here we consider sheaves of $\cor$-modules for a commutative unital ring $\cor$. 
Roughly speaking, a codirection $(x_0;\xi_0)\in T^*M$ does not belong to $\musupp(F)$ if for any smooth function $\phi\cl M\to\R$ such that $\phi(x_0)=0$ and $d\phi(x_0)=\xi_0$, any section of $H^j(\{\phi<0\};F)$ extends uniquely in a neighborhood of $x_0$. In other words, $(\rsect_{\{\phi\geq0\}}(F))_{x_0}\simeq0$. 
The microsupport of a sheaf describes the codirections in which it is not locally constant 
and a sheaf whose microsupport is contained in the zero-section is nothing but a locally constant sheaf.
Here again, a local notion (that of being locally constant) becomes a global notion with respect to the projection $T^*M\to M$. 

One can give natural bounds to the microsupport of sheaves after the six operations, in particular after proper direct images and 
non-characteristic inverse images. The formulas one obtains are formally similar to the classical ones  for D-modules.
As an application, one gets a generalization of Morse theory. Indeed, in the classical setting this theory asserts that given a proper function $\phi\cl M\to\R$, the topology of the set $M_t=\{x\in M;\phi(x)<t\}$ does not change as far as $t$ does not meet a critical value of $\phi$, that is, $\rsect(M_t;\cor_M)$ is constant in $t$ on an interval $(t_0,t_1)$ in which there is no critical values. Now we can replace the constant sheaf $\cor_M$ with any sheaf $F$ on $M$, a critical value of $\phi$ becoming any 
$t=\phi(x)\in\R$ such that $d\phi(x)\in\musupp(F)$. Here, we shall interpret the new notion of  ``barcodes'' (see~\cite{Gri08}) in this setting.

The main property of the microsupport is that it is a co-isotropic ({\em i.e.,} involutive) subset of the symplectic manifold $T^*M$. This is the reason why this theory has many applications in symplectic topology, as we shall see in Section~\ref{section:applications}.

\subsection{Microsupport}
Let $M$ be a real manifold, say of class $C^\infty$.

\begin{definition}[{See~\cite[Def.~5.1.2]{KS90}}]\label{def:SS}
Let $F\in \Derb(\cor_M)$. One denotes by $\musupp(F)$ the closed subset of $T^*M$ defined as follows. For an open subset $U\subset T^*M$, $U\cap\musupp(F)=\varnothing$ if and only if for  any  $x_0\in M$
and any real $C^1$-function $\phi$ on $M$ defined in a neighborhood of $x_0$ 
satisfying $d\phi(x_0)\in U$ and $\phi(x_0)=0$, one has
$(\rsect_{\{x;\phi(x)\geq0\}} (F))_{x_0}\simeq0$. One calls $\musupp(F)$ the microsupport of $F$.
\end{definition}
In other words, $U\cap\musupp(F)=\varnothing$ if the sheaf $F$ has no cohomology 
supported by ``half-spaces'' whose conormals are contained in $U$. 

In the sequel, we denote by $T^*_MM$ the zero-section of $T^*M$, identified to $M$.

\begin{itemize}
\item
By its construction, the microsupport is closed and is
conic, that is, invariant by the action of  $\R^+$ on $T^*M$. 
\item
$\musupp(F)\cap T^*_MM =\pi_M(\musupp(F))=\supp(F)$. 
\item
$\musupp(F)=\musupp(F\,[j])$ ($j\in\Z$).
\item
The microsupport satisfies the triangular inequality:
if $F_1\to F_2\to F_3\to[+1]$ is a
distinguished triangle in  $\Derb(\cor_M)$, then 
$\musupp(F_i)\subset\musupp(F_j)\cup\musupp(F_k)$ for all $i,j,k\in\{1,2,3\}$
with $j\not=k$. 
\end{itemize}

\begin{example}\label{exa:SSi}
(i)  $\musupp(F)\subset T^*_MM$ if and only if  $H^j(F)$ is locally constant on $M$ for all $j\in\Z$.

\noindent
(ii) If $N$ is a smooth closed submanifold of $M$, then 
$\musupp(\cor_N)=T^*_NM$, the conormal bundle to $N$ in $M$.

\noindent
(iii) The link between the microsupport of sheaves and the characteristic variety of D-modules will be given in Theorem~\ref{th:ssinchar1}. 

\noindent
(iv) Let $\phi$ be $C^1$-function with $d\phi(x)\not=0$ when $\phi(x)=0$.
Let $U=\{x\in M;\phi(x)>0\}$ and let $Z=\{x\in M;\phi(x)\geq0\}$. 
Then, 
\eqn
&&\musupp(\cor_U)=U\times_MT^*_MM\cup\{(x;\lambda d\phi(x));\phi(x)=0,\lambda\leq0\},\\
&&\musupp(\cor_Z)=Z\times_MT^*_MM\cup\{(x;\lambda d\phi(x));\phi(x)=0,\lambda\geq0\}.
\eneqn
\centerline{
\includegraphics[height=1.5in]{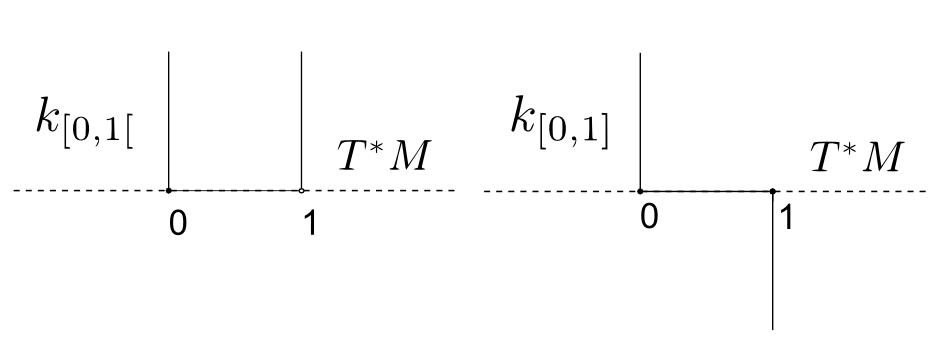}}
In these pictures, $M=\R$ and $T^*M=\R^2$. 
\end{example}

\subsubsection*{Co-isotropic subsets}
The map $\pi_M\cl T^*M\to M$ induces the maps 
\eqn
&&\xymatrix{
T^*T^*M& T^*M\times_M T^*M\ar[l]_-{\pi_{Md}}\ar[r]^-{\pi_{M\pi}}& T^*M
}\eneqn
(see Diagram~\ref{diag:cotgmor}  below). 
By composing  the map $\pi_{Md}$ with  the diagonal map $T^*M\into T^*M\times_M T^*M$, we get a map 
$\alpha_M\cl T^*M\to T^*T^*M$, that is, a section of $T^*(T^*M)$. This is
the Liouville $1$-form, given in a local homogeneous symplectic coordinate system 
$(x;\xi)$ on $T^*M$, by 
\eqn
&&\alpha_M=\sum_{j=1}^n\xi_j \,dx_j. 
\eneqn
The differential $d\alpha_M$ of the Liouville form is the symplectic form $\omega_M$ on $T^*M$ 
given in a local symplectic coordinate system 
$(x;\xi)$   on $T^*M$ by
$\omega_M=\sum_{j=1}^n d\xi_j\wedge dx_j$.
Hence $T^*M$ is not only a symplectic manifold, it is a homogeneous (or exact) symplectic manifold. 

The form $\omega_M$ induces an isomorphism 
$H\cl T^*(T^*M)\isoto T(T^*M)$ called the Hamiltonian isomorphism.  In a local symplectic coordinate system $(x;\xi)$,
this isomorphism  is given by
\eqn
&&H(\langle\lambda,dx\rangle +\langle \mu,d\xi\rangle)
=-\langle\lambda,\partial_\xi\rangle +\langle \mu,\partial_x\rangle.
\eneqn

\begin{definition}[{See~\cite{KS90}*{Def. 6.5.1}}]\label{def:coisotropic}
A subset $S$ of $T^*M$ is co-isotropic \lp one also says involutive\rp\, 
at $p\in T^*M$ if $C_p(S,S)^\perp\subset C_p(S)$. Here we identify the orthogonal $C_p(S,S)^\perp$ to a subset of $T_pT^*M$ via the Hamiltonian isomorphism.
\end{definition}
When $S$ is smooth, one recovers the usual notion.

\begin{example}
Let $M=\R$ and denote by $(t;\tau)$ the coordinates on $T^*M$. The set $\{(t;\tau);t\geq0,\tau=0\}$ is not co-isotropic, contrarily 
to the set $\{(t;\tau);t\geq0,\tau=0\cup t=0,\tau\geq0\}$ which is co-isotropic.
\end{example}

An essential property of the microsupport is given by the next theorem. 
\begin{theorem}[{See~\cite[Th.~6.5.4]{KS90}}]\label{th:sscoisotropic}
Let $F\in \Derb(\cor_M)$. Then its microsupport 
$\musupp(F)$ is co-isotropic.
\end{theorem} 

\subsubsection*{Constructible sheaves}

Assume that $M$ is real analytic and $\cor$ is a field. We do not recall here the definition of a subanalytic subset and a subanalytic stratification, referring to~\cite{BM88}. 
\begin{definition}\label{def:construct}
A sheaf $F$ is weakly $\R$-constructible if there exists a subanalytic stratification $M=\bigsqcup_\alpha M_\alpha$ such that for each strata $M_\alpha$, the restriction $F\vert_{M_\alpha}$ is locally constant. 
If moreover,  it is a local system (i.e., is locally constant of finite rank), then one says that $F$ is  $\R$-constructible. 
\end{definition}
One denotes by $\Derb_{\wRc}(\cor_{M})$ (resp.\ $\Derb_{\Rc}(\cor_{M})$) the full subcategory of $\Derb(\cor_{M})$ consisting of sheaves with weakly $\R$-constructible cohomology (resp.\  $\R$-constructible cohomology).

A subanalytic $\R^+$-conic subset $\Lambda$ of $T^*M$ is isotropic if the $1$-form $\alpha_M$ vanishes on $\Lambda$.  It is Lagrangian if it is both isotropic and co-isotropic. 

\begin{theorem}[{See~\cite{KS90}*{Th.~8.4.2}}]\label{th:construct}
Assume that $M$ is real analytic and $\cor$ is  a field. 
Let $F\in\Derb(\cor_M)$. Then $F\in\Derb_{\wRc}(\cor_{M})$ if and only if $\musupp(F)$ is contained in a closed $\R^+$-conic 
subanalytic Lagrangian subset of $T^*M$ and this implies that $\musupp(F)$ is itself a closed $\R^+$-conic 
subanalytic Lagrangian subset of $T^*M$.
\end{theorem}

\subsection{Microsupport and the six operations}
Let $f\cl M\to N$ be  a morphism of  real manifolds. The tangent map $Tf\cl TM\to TN$ decomposes as 
$TM\to M\times_NTN\to TN$ and by duality, one gets the diagram: 
\eq\label{diag:cotgmor}
&&\xymatrix{
T^*M\ar[rd]_-{\pi_M}&M\times_N\ar[d]^-\pi\ar[l]_-{f_d}\ar[r]^-{f_\pi}T^*N
                                      & T^*N\ar[d]^-{\pi_N}\\
&M\ar[r]^-f&N.
}\eneq
One sets 
\eqn
&&T^*_MN\eqdot\ker f_d= \opb{f_d}(T^*_MM). 
\eneqn
Note that, denoting by $\Gamma_f$ the graph of $f$ in $M\times N$, 
the projection $T^*(M\times N)\to M\times T^*N$ identifies 
$T^*_{\Gamma_f}(M\times N)$ and $M\times_NT^*N$.
\begin{definition}
Let $\Lambda\subset T^*N$ be a closed $\R^+$-conic subset. One says that 
$f$ is non-characteristic for $\Lambda$ if
\eqn
&&\opb{f_\pi}\Lambda\cap T^*_MN\subset M\times_NT^*_NN.
\eneqn
\end{definition}
This is equivalent to saying that $f_d$ is proper on $\opb{f_\pi}\Lambda$.

\begin{theorem}\label{th:oim}
Let $f\cl M\to N$ be a morphism of manifolds,
let $F\in\Derb(\cor_M)$ and assume that $f$ is proper on $\supp(F)$. Then $\reim{f}F\isoto\roim{f}F$
and 
\eq\label{eq:ssoim}
&&\musupp(\roim{f}F)\subset f_\pi\opb{f_d}\musupp(F).
\eneq 
Moreover, if $f$ is a closed embedding, this inclusion is an equality. 
\end{theorem}
This result may be interpreted as a ``stationary phase lemma'' for sheaves, or else as a generalization of Morse theory for sheaves (see below).

\centerline{
\includegraphics[height=1.5in]{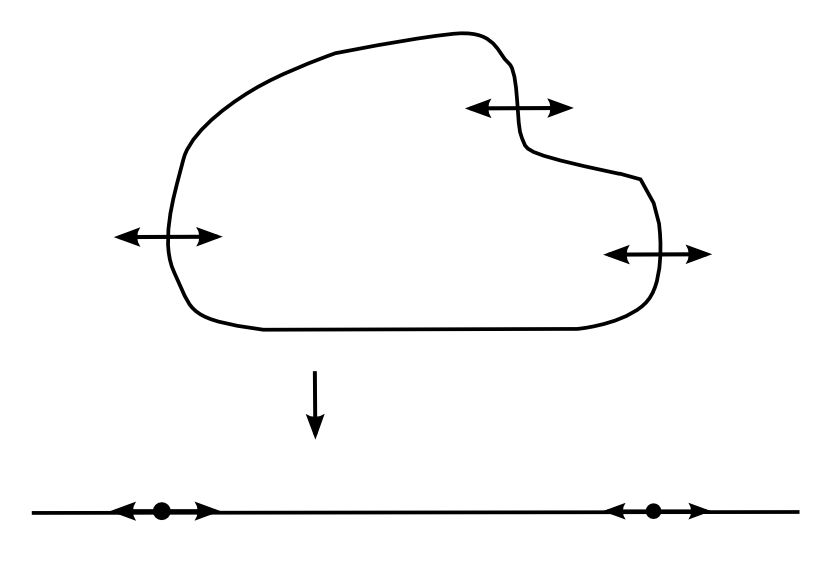}}
On this picture, one represents the direct image of the constant sheaf on the contour.  The microsupport of the direct image is contained in the image of the ``horizontal'' conormal vectors. 
This shows that the inclusion in Theorem~\ref{th:oim} may be strict.

\begin{theorem}\label{th:opb}
Let $f\cl M\to N$ be a morphism of manifolds, let $G\in\Derb(\cor_N)$ and 
assume that $f$ is non-characteristic with respect to $\musupp(G)$. Then
the natural morphism $\opb{f}G\tens\, \omega_{M/N}\to\epb{f}G$
is an isomorphism and 
 \eq\label{eq:ssopb}
 &&\musupp(\opb{f}G)\subset f_d\opb{f_\pi}(\musupp(G)).
 \eneq
 Moreover, if $f$ is submersive, this inclusion is an equality. 
\end{theorem}
Note that the formulas for the microsupport of the direct or inverse images  is analogue to those for the 
direct or inverse images of D-modules.

\begin{corollary}\label{cor:opboim}
Let $F_1,F_2\in\Derb(\cor_M)$.
\bnum
\item
Assume that $\musupp(F_1)\cap\musupp(F_2)^a\subset T^*_MM$. Then
\eqn
&&\musupp(F_1\lltens F_2)\subset \musupp(F_1)+\musupp(F_2).
\eneqn
\item
Assume that $\musupp(F_1)\cap\musupp(F_2)\subset T^*_MM$. Then
\eqn
&&\musupp(\rhom(F_2,F_1))\subset \musupp(F_2)^a+\musupp(F_1).
\eneqn
\enum
\end{corollary}
Note that the formula for the microsupport of the tensor product is analogue to that giving a bound to the wave front set of the product of two distributions $u_1$ and $u_2$ satisfying $\WF(u_1)\cap\WF(u_2)^a\subset \sqrt{-1}T^*_MM$. 

\begin{corollary}[{A kind of Petrowsky theorem for sheaves}]\label{cor:petrovsky}
Let $F_1,F_2\in\Derb(\cor_M)$. Assume that $F_2$ is constructible and $\musupp(F_2)\cap\musupp(F_1)\subset T^*_MM$.
Then the natural morphism 
\eq\label{eq:elliso1}
&&\RD_M'F_2\tens F_1\to\rhom(F_2,F_1).
\eneq
is an isomorphism. 
\end{corollary}
The link with the classical Petrowsky theorem for elliptic operators will be given in Corollary~\ref{cor:petrowshv3}.

\begin{remark}
(i) One can also give bounds to the microsupports of the sheaves obtained by the six operations without assuming any hypothesis of properness or transversality. See~\cite{KS90}*{Cor.~6.4.4,~6.4.5}.

\spa
(ii) By applying the results on the behaviour of the microsupport together with Theorem~\ref{th:construct}, one recovers easily the fact that the category of $\R$-constructible sheaves is stable with respect to the six operations (under suitable hypotheses of properness). See~\cite{KS90}*{\S~8.4}.
\end{remark}

\subsubsection*{Morse theory}
As an application of Theorem~\ref{th:oim}, one gets:
\begin{theorem}\label{th:Morse}
Let $F\in\Derb(\cor_M)$, let $\psi\cl M\to\R$ be a function of class $\Cd^1$ and assume that $\psi$
is proper on $\supp(F)$. Let $a,b\in\R$ with $a<b$
and assume that $d\psi(x)\notin\musupp(F)$ for $a\leq \psi(x)<b$. 
For $t\in\R$, set $M_t=\opb{\psi}(]-\infty,t[)$. Then the
restriction morphism  $\rsect(M_b;F)\to\rsect(M_a;F)$ is an isomorphism.
\end{theorem}
The classical Morse theorem corresponds to the constant sheaf $F=\cor_M$.

As an immediate corollary, one obtains:
\begin{corollary}\label{cor:Morse1}
Let $F\in\Derb(\cor_M)$ and let $\psi\cl M\to\R$ be a function of class $\Cd^1$. Set $\Lambda_\psi=\{(x;d\psi(x))\}$, a \lp non-conic in general\rp\, Lagrangian submanifold of $T^*M$.  Assume that 
$\supp(F)$ is compact
and
$\rsect(M;F)\not=0$. Then $\Lambda_\psi\cap\musupp(F)\not=\varnothing$.
\end{corollary}
This corollary is an essential tool in a new proof of the Arnold non-displaceability theorem (see~\S~\ref{subsection:symp}).

\subsubsection*{Persistent homology and barcodes}
Persistent homology and barcodes are recent concrete applications of algebraic topology. They can easily be interpreted in the language of microlocal sheaf theory, as follows.

Consider a finite set $X=\bigcup_{i\in I}\{x_i\}\subset\R^n$ (a cloud). In order to understand its topology, one replaces each point $x_i\in X$ with a closed ball 
$B(x_i;t)$ of radius $t$ and looks at the topology of the set $X_t=\bigcup_{i\in I}B(x_i;t)$, more precisely one looks how this topology changes when $t$ goes from $0$ to $\infty$.
Set
\eqn
&&Z=\bigcup_{t\geq0} X_t\times\{t\}\subset\R^{n+1},
\eneqn
denote by $(x,t)$ the coordinates on $\R^{n+1}$,  by $(x,t;\xi,\tau)$ the associated coordinates on $T^*\R^{n+1}$ 
and  by $p\cl\R^{n+1}\to\R$ the projection $p(x,t)=t$. 

Let $\cor$ be a field and  consider the sheaf $\cor_Z$. Clearly, this sheaf is $\R$-constructible. Since the map $p$ is proper on $Z$, we get that $\roim{p}\cor_Z\in\Derb_\Rc(\cor_\R)$. Moreover, one has
\eq
&&\musupp(\cor_Z)\subset \{(x,t;\xi,\tau);\tau\geq0\}.
\eneq
To check it, one can argue by induction on the cardinality of $I$, using Corollary~\ref{cor:opboim}. 
In fact, this result holds true in a more general situation, replacing $X$ with a compact set and $\R^n$ with a Riemannian manifold. This follows from~\cite{KS90}*{Prop.~5.3.8}.

Since $p$ is proper on $Z$, we get by Theorem~\ref{th:oim} that
\eq\label{eq:taupositif}
&&\musupp(\roim{p}\cor_Z)\subset \{(t,\tau);\tau\geq0\}.
\eneq
A classical result of sheaf theory (see~\cite{Gu16}*{\S~6}) asserts that any $G\in\Derb(\cor_\R)$ is a finite direct sum of constant sheaves on intervals. 
By~\eqref{eq:taupositif}, one gets that there are finite many  intervals $[a_j,b_j)$, $a_j\geq0$, $a_j<b_j$, $0<b_j\leq\infty$ and integers 
$d_j\in\Z$ ($j\in J$, $J$ finite) such that 
\eq\label{eq:barcod}
&&\roim{p}\cor_Z\simeq \bigoplus_{j\in J}\cor_{[a_j,b_j)}\,[d_j].
\eneq
Here we may have $a_j=a_k, b_j=b_k, d_j=d_k$ for $j\neq k$. 
To the right-hand side of~\eqref{eq:barcod} we may associate a family of barcodes as follows. 
For each $d_j\in\Z$,  replace $[a_j,b_j)\,[d_j]$ with the vertical interval $[0,b_j)$ centered at $a_j\in\R$ (see~\cite{Gri08}).

\subsection{The functor $\muhom$}\label{section:muhom}
In~\cite{KS90}*{\S~IV.4}, Sato's microlocalization functor is generalized as follows.
The  functor $\muhom\cl \Derb(\cor_M)^\rop\times \Derb(\cor_M)\to\Derb(\cor_{T^*M})$
is given by
\eqn
&&\muhom(F_2,F_1)=\opb{\tw\delta}\mu_\Delta\rhom(\opb{q_2}F_2,\epb{q_1}F_1)
\eneqn
where $q_i$ ($i=1,2$) denotes the $i$-th projection on  $M\times M$,  
 $\Delta$ is the diagonal of $M\times M$ and  $\tw\delta$ the isomorphism
$\tw\delta\cl T^*M\isoto T^*_M(M\times M),\quad (x;\xi)\mapsto (x,x;\xi,-\xi)$.
One proves that:
\begin{itemize}
\item
$\roim{\pi_M}\muhom(F_2,F_1)\simeq\rhom(F_2,F_1)$,
\item
$\muhom(\cor_N,F)\simeq\mu_N(F)$  for $N$ a closed submanifold of $M$,
\item
$\supp\muhom(F_2,F_1)\subset\musupp(F_1)\cap\musupp(F_2)$.
\end{itemize}

The functor $\muhom$ is  the functor of microlocal morphisms. Let us make this assertion more precise. 

Let $Z$ be a locally closed  subset of $T^*M$. One denotes by 
$\Derb(\cor_M;Z)$ the localization of $\Derb(\cor_M)$ by its full triangulated subcategory consisting of objects $F$ such that $\musupp(F)\cap Z=\varnothing$. The objects of $\Derb(\cor_M;Z)$ are those of $\Derb(\cor_M)$ but a
morphism $u\cl F_1\to F_2$ in $\Derb(\cor_M)$ becomes an isomorphism
in $\Derb(\cor_M;Z)$ if, after embedding  this morphism in a distinguished
triangle $F_1\to F_2\to F_3\to[+1]$, one has $\musupp(F_3)\cap Z=\varnothing$. 

One shall be aware that the prestack (a prestack is, roughly speaking, a presheaf of categories)
$U\mapsto \Derb(\cor_M;U)$ ($U$ open in $T^*M$)
is not a stack, not even a separated prestack. 
The functor $\muhom$ induces a bifunctor (we keep the same notation):
\eqn
&&\muhom\cl\Derb(\cor_M;U)^\rop\times\Derb(\cor_M;U)\to\Derb(\cor_{U}).
\eneqn
and for $F_1, F_2, F_3\in\Derb(\cor_M;U)$,  there is a natural morphism 
\eq\label{eq:compmuhom}
&&\muhom(F_3,F_2)\ltens\muhom(F_2,F_1)\to\muhom(F_3,F_1).
\eneq
Moreover, for $p\in T^*M$, one has an isomorphism 
\eq
&&\muhom(F_2,F_1)_p\simeq\Hom[{\Derb(\cor_M;\{p\})}](F_2,F_1)
\eneq
and~\eqref{eq:compmuhom} is compatible with the composition of morphisms in $\Derb(\cor_M;\{p\})$. 
This shows that, in some sense, $\muhom$ is kind of internal hom for  $\Derb(\cor_M;U)$.

Now let $\Lambda$ be a smooth conic submanifold closed in $U$. Denote by 
$\Derb_\Lambda(\cor_M;U)$ the full subcategory of $\Derb(\cor_M;U)$ consisting of objects 
$F\in\Derb(\cor_M)$ satisfying $\musupp(F)\cap U\subset\Lambda$.  
\begin{definition}\label{def:simpleshv}
Assume  that $\cor$ is a field. 
One says that $F\in \Derb_\Lambda(\cor_M;U)$ is simple along $\Lambda$ if the natural morphism 
$\cor_\Lambda\to\muhom(F,F)$ is an isomorphism.  One denotes by  $\Simple(\Lambda,\cor)$ the  subcategory of  $\Derb_\Lambda(\cor_{M};U)$ consisting of simple sheaves. 
\end{definition}
We shall see in~\S~\ref{subsection:symp} that simple sheaves play an important role in symplectic topology.

\subsubsection*{Microlocal Serre functor}
Recall first what a Serre functor is, a notion introduced in~\cite{BK89}. Consider linear triangulated category $\sht$  over a field $\cor$ and assume that the spaces $\bigoplus_{n\in\Z}\Hom[\sht](A,B\,[n])$ are finite dimensional for any $A,B\in\sht$. 
A Serre functor $S$ is an endofunctor $S$ of $\sht$ together with a  functorial (in $A$ and $B$) isomorphism :
\eqn
&&\Hom[\sht](A,B)^*\simeq\Hom[\sht](B,S(A)).
\eneqn
Here ${}^*$ denotes the duality functor for vector spaces. 

This definition is motivated by the example of the category $\Derb_\coh(\sho_X)$ of coherent $\sho_X$-modules on a complex compact manifold $X$. In this case, a theorem of Serre asserts that 
the Serre functor is given by $\shf\mapsto\shf\tens[\sho_X]\Omega_X\,[d_X]$, where $\Omega_X$ is the sheaf of holomorphic forms of maximal degree and $d_X$ is  the complex dimension of $X$.

There is an interesting phenomenon which holds with $\muhom$ and not with $\rhom$.  Indeed,  although the category $\Derb_{\Rc}(\cor_{M})$  
does not admit a Serre functor, it admits 
a kind of microlocal Serre functor, as shown by the isomorphism, functorial in $F_1$ and $F_2$ 
(see~\cite[Prop.~8.4.14]{KS90}):
\eqn
&&\RD_{T^*M}\muhom(F_2,F_1)\simeq\muhom(F_1,F_2)\tens\opb{\pi_M}\omega_M.
\eneqn
Here, $\omega_M\simeq\ori_M\,[\dim M]$ is the dualizing complex on $M$. 

This confirms the fact that to fully understand $\R$-constructible sheaves, it is natural to look at them 
microlocally, that is, in $T^*M$. This is also in accordance with the ``philosophy'' of Mirror Symmetry which interchanges the category of coherent $\sho_X$-modules on a complex manifold $X$ with the Fukaya category on a symplectic manifold $Y$. 
(See \S~\ref{subsection:symp}).

\section{Some applications}\label{section:applications}

\subsection{Solutions of D-modules}\label{subsection:Dmod}
 We shall first briefly present  some applications of microlocal sheaf theory to systems of linear partial differential equations (LPDE), that is, (generalized) holomorphic solutions of D-modules. This was the original motivation of the theory. The main tool 
is a theorem which asserts  that given a coherent D-module $\shm$ on a complex manifold $X$, the microsupport of the complex of the holomorphic solutions of $\shm$ is contained in  the characteristic variety of the system.  (In fact it is equal, but the other inclusion is not so useful.) This theorem is deduced from the Cauchy-Kowalevska theorem in its precise form given by Petrowsky and Leray, and is the unique tool of analysis which is used thereafter. With this result, the study of generalized solutions of systems of LPDE reduces most of the time  to a geometric study, the relations between the microsupport of the constructible sheaf associated with the space of generalized functions ({\em e.g.,} the conormal bundle $T^*_MX$ to a real manifold $M$) and the characteristic variety of the system. 
We shall only study here the particular case of elliptic systems. 

\subsubsection*{D-modules}
We have seen at the end of~\S~\ref{subsection:abel} that a system of linear equations over a sheaf of rings $\shr$ is nothing but an $\shr$-module locally of finite presentation.
We shall consider here the case where $X$ is a complex manifold and $\shr$ is the sheaf $\shd_X$ of holomorphic differential operators. References for D-modules are made to~\cite{Ka70, Ka03}.

Let $X$ be a complex manifold.
One denotes by $\shd_X$  the sheaf of rings of holomorphic (finite order) differential
operators. It is a right and left coherent ring. A system of linear partial differential equations on $X$ is   thus a 
left coherent $\shd_X$-module $\shm$. 
Locally on $X$, $\shm$ may be represented as the cokernel 
of a matrix $\cdot P_0$ of differential operators acting on the right: 
\eqn
&&\shm\simeq \shd_X^{N_0}/\shd_X^{N_1}\cdot P_0.
\eneqn
By classical arguments of analytic geometry (Hilbert's syzygy
theorem), one shows that 
$\shm$ is locally isomorphic to the cohomology of a bounded complex
\eq\label{eq:globalpresent}
&&\shm^\bullet\eqdot
0\to \shd_X^{N_r}\to\cdots\to\shd_X^{N_1}\to[\cdot P_0]\shd_X^{N_0}\to 0.
\eneq
Clearly, the sheaf $\sho_X$ of holomorphic functions is a left $\shd_X$-module. It is coherent since $\sho_X\simeq\shd_X/\shi$ where $\shi$ is the left ideal generated by the vector fields. 
For  a coherent $\shd_X$-module $\shm$, one sets for short 
\eqn
&&\Sol(\shm)\eqdot\rhom[\shd_X](\shm,\sho_X).
\eneqn
Representing (locally) $\shm$ by a bounded complex $\shm^\bullet$ as above, 
we get
\eq\label{eq:solm}
&&\Sol(\shm)\simeq
0\to \sho_X^{N_0}\to[P_0\cdot]\sho_X^{N_1}\to\cdots\sho_X^{N_r}\to 0,
\eneq
where now $P_0\cdot$ operates on the left.

\subsubsection*{Characteristic variety}

One defines naturally the characteristic variety of $\shm$, 
denoted $\chv(\shm)$, a closed complex analytic subset of $T^*X$, 
conic with respect to the action of $\C^\times$ on $T^*X$.
For example, if $\shm$ has a single generator $u$ with relation $\shi u=0$, where
$\shi$ is a locally finitely generated left ideal of $\shd_X$, then 
\eqn
&&\chv(\shm)=\{(z;\zeta)\in T^*X; \sigma(P)(z;\zeta)=0\mbox{ for all }P\in\shi\},
\eneqn
where $\sigma(P)$ denotes the principal symbol of $P$.

The fundamental  result below was first obtained in~\cite{SKK73}.
\begin{theorem}
Let $\shm$ be a coherent $\shd_X$-module. Then 
$\chv(\shm)$ is a closed conic  complex analytic {\em involutive} \lp {\em i.e.,} co-isotropic\rp\, subset of $T^*X$. 
\end{theorem}
The proof  of the involutivity is really difficult:  it uses microdifferential operators of infinite order and quantized contact transformations. 
Later, Gabber \cite{Ga81} gave a purely algebraic (and much simpler) proof of this result. 
Theorem~\ref{th:ssinchar1} below together with Theorem~\ref{th:sscoisotropic} gives another totally different proof of the involutivity.

\begin{theorem}[{See~\cite[Th.~11.3.3]{KS90}}]\label{th:ssinchar1}
Let $\shm$ be a coherent $\shd_X$-module. Then 
\eq\label{eq:SS=chv}
&&\musupp(\Sol(\shm))=\chv(\shm).
\eneq
\end{theorem}
The only analytic  tool in the proof of the inclusion $*\subset *$ in~\eqref{eq:SS=chv}
is the classical Cauchy-Kowalevska theorem, in its precise form (see~\cite{Ho83}*{\S~9.4}). To prove the reverse inclusion, one uses a theorem of~\cite{SKK73} which asserts that the ring $\shd_X^\infty$  of infinite order differential operators is faithfully flat over $\shd_X$.

\subsubsection*{Elliptic pairs}

Let us apply Corollary~\ref{cor:petrovsky}  when $X$ is a complex manifold.
For $G\in\Derb_\Rc(\C_X)$, set
\eqn
&&\sha_G\eqdot\sho_X\tens G, \quad \shb_G\eqdot\rhom(\RD_X'G,\sho_X).
\eneqn
Note that if $X$ is the complexification of a real analytic manifold $M$ and we choose $G=\C_M$, we 
recover the sheaf of real analytic functions and the sheaf of hyperunctions:
\eqn
&&\sha_{\C_M}=\sha_M,\quad \shb_{\C_M}=\shb_M.
\eneqn
Now let $\shm\in\Derb_\coh(\shd_X)$. 
According to~\cite{ScSn94}, one says that the pair $(G,\shm)$ is elliptic if 
$\chv(\shm)\cap\musupp(G)\subset T^*_XX$. 

\begin{theorem}[{\cite{ScSn94}}]  \label{th:petrowshv2}
Let $(\shm,G)$ be an elliptic pair. 
\banum
\item
We have the canonical isomorphism:
\eq\label{eq:solasolG}
&&\rhom[\shd_X](\shm,\sha_G)\isoto\rhom[\shd_X](\shm,\shb_G).
\eneq
\item
Assume moreover that $\supp(\shm)\cap\supp(G)$ is compact and $\shm$ admits a global presentation as in~\eqref{eq:globalpresent}.
Then the cohomology of the complex 
$\RHom[\shd_X](\shm,\sha_G)$
is finite dimensional.
\eanum
\end{theorem}
\begin{proof}
(a) This is a particular case of Corollary~\ref{cor:petrovsky}.

\spa
(b) One first shows that one may represent $G$ with a bounded complex whose components are  a finite direct sum of sheaves of the type 
$\cor_U$ for $U$ open subanalytic relatively compact in $X$. 
Then one can represent  the left hand side of
the global sections of \eqref{eq:solasolG}
by a complex of topological vector spaces of type DFN and the right hand side 
by a complex of topological vector spaces of type FN. The finiteness follows by classical results of functional analysis. 
\end{proof}

Let us particularize Theorem~\ref{th:petrowshv2} to the usual case of an elliptic system. Let $M$ be a real anaytic manifold, $X$ a complexification of $M$ and let us choose  $G=\RD'_X\C_M$. Then $(G,\shm)$ is an elliptic pair if and only if
\eq\label{eq:elliptic}
&&T^*_MX\cap\chv(\shm)\subset T^*_XX.
\eneq
In this case, one simply says that $\shm$ is an elliptic system. Then one recovers a classical result:

\begin{corollary}  \label{cor:petrowshv3}
Let $\shm$ be an elliptic system.
\banum
\item
We have the canonical isomorphism:
\eq\label{eq:solasolM}
&&\rhom[\shd_X](\shm,\sha_M)\isoto\rhom[\shd_X](\shm,\shb_M).
\eneq
\item
Assume moreover that $M$ is compact and $\shm$ admits a global presentation as in~\eqref{eq:globalpresent}.
Then the cohomology of the complex 
$\RHom[\shd_X](\shm,\sha_M)$
is finite dimensional.
\eanum
\end{corollary}
There is a more precise result than Corollary~\ref{cor:petrowshv3}~(a), due to Sato~\cite{Sa70}.
\begin{proposition}\label{pro:WF}
Let $U\subset T^*_MX$ be an open subset, let $\shm$ be a coherent $\shd_X$-module, let $j\in\Z$ and let $u\in\sect(U;\ext[\shd_X]{j}(\shm,\shc_M))$. Then 
$\supp(u)\subset U\cap\chv(\shm)$. In particular, if $u\in\Hom[\shd_X](\shm,\shb_M)$, then
$\WF(u)\subset  T^*_MX\cap\chv(\shm)$. 
\end{proposition}
\begin{proof}
One has
\eqn
\rhom[\opb{\pi_M}\shd_X](\opb{\pi_M}\shm,\shc_M)&\simeq&\muhom(\RD'\C_M,\rhom(\shm,\sho_X))
\eneqn
and the support of the right-hand side is contained in $\musupp(\rhom(\shm,\sho_X))\cap\musupp(\C_M)$, that is, in 
$T^*_MX\cap\chv(\shm)$. 
\end{proof}
In case $\shm$ is elliptic, we get that $\WF(u)$ is contained in the zero-section, hence $u$ is real analytic. 
\begin{remark}
One can treat similarly hyperbolic systems and in particular, one can solve globally the Cauchy problem on globally hyperbolic spacetimes, using only tools from sheaf theory (see~\cite{JS16}).
\end{remark}

\subsection{A glance at symplectic topology}\label{subsection:symp}

When a space is endowed with a certain structure, it is natural to associate to it a sheaf (or something similar, a stack for example) which takes into account this structure. On a real manifold, one considers the sheaf of $C^\infty$-functions, on a complex manifold the sheaf of holomorphic functions, on a complex symplectic manifold the sheaves of holomorphic deformation-quantization algebras (which are  sheaves only locally, globally one has to replace the notion of a sheaf by that of  an algebroid stack). Then, on a real symplectic manifold $\symx$, what is, or what are, the candidate(s)?
One answer is not given by a sheaf but by a category, the ``Fukaya category''. We shall not describe it here, let us only mention that its construction is not local (and cannot be so), its objects are smooth Lagrangian submanifolds (plus some data such as local systems) and the morphisms between two Lagrangians are only defined when the Lagrangians are transverse (which makes it difficult to define the identity morphisms!). 

It is only quite recently that people realized that microlocal sheaf theory is an efficient tool to treat many questions of  symplectic topology.  In~\cite{Ta08} Dmitri Tamarkin gives a new proof of the Arnold non-disp\-laceabi\-lity theorem and in~\cite{Na09, NZ09}  David Nadler and Eric Zaslow  showed that the Fukaya category on $T^*M$ is equivalent (in some sense) to the bounded derived category of $\R$-constructible sheaves on $M$. 
These works opened new perspectives  in symplectic topology 
(see in particular~\cite{Gu12, Gu13, Gu16})  and also in the study of  Legendrian knots (see~\cite{STZ14}).

Note that, for $U$ open in $T^*M$, a substitute to the Fukaya category on $U$ could be  the triangulated category  $\Derb(\cor_M;U)$ already encountered.  However, if the objects of this category are associated with the microsupports of sheaves, which are  co-isotropic, these microsupports are $\R^+$-conic. To treat non-conic Lagrangian submanifolds, Tamarkin develops a kind of no more conic microlocal sheaf theory by adding a variable $t$, with dual variable $\tau$, and works in the category of sheaves on $M\times\R$ localized at $\tau>0$.

\subsubsection*{Arnold non-displaceability theorem}

Let us explain the classical  Arnold non-displaceability  conjecture, which has been a theorem for long. 

A symplectic isotopy of a symplectic manifold $\symx$ is a $1$-parameter family of isomorphisms of $\symx$ which respect the symplectic structure. More precisely, $I$ is an open interval containing $0$ and $\Phi\cl \symx \times I\to \symx$ is a $C^\infty$-map such that
$\phi_t\eqdot\Phi(\cdot,t)\cl \symx \to \symx$ is a symplectic isomorphism for
each $t\in I$ and is the identity for $t=0$. One says that the isotopy is Hamiltonian if the graph of 
$\Phi$ in  $\symx \times I\times\symx$ is the projection of a Lagrangian submanifold $\Lambda_\Phi$ of $\symx \times T^*I\times\symx$.  
Then Arnold's conjecture says that for $\symx=T^*N$ where $N$ is  a compact $C^\infty$-manifold, 
$\phi_t(T^*_NN)\cap T^*_NN\neq\varnothing$ for all $t\in I$. 

\centerline{
\includegraphics[height=1.5in]{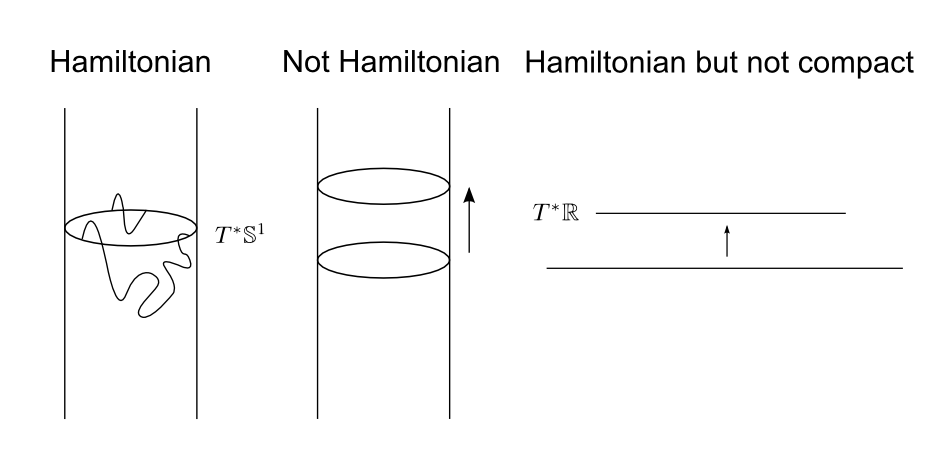}}
As already mentioned, Tamarkin~\cite{Ta08} has given a totally new proof of this result by adapting microlocal sheaf theory to a non-conic setting. However,  one can also deduce Arnold's conjecture from another conjecture which is itself ``conic''. This is the strategy of ~\cite{GKS12} that we shall expose. 

Notice first that a homogeneous symplectic isotopy of an open conic subset $U\subset T^*M$ is a  symplectic isotopy which commutes with the $\R^+$-action. In such a case, the isotopy is automatically Hamiltonian.

Recall that $ \dTM$ denotes the bundle $T^*M$ with the zero-section removed. One denotes by $\Derlb(\cor_{M})$ the full subcategory of $\RD(\cor_M)$ consisting of locally bounded objects. 
For $K\in\Derlb(\cor_{M\times M\times I})$ and $s\in I$, we shall denote by $K_s$ the restriction of $K$ to
$M\times M\times\{s\}$.

\begin{theorem} {[Quantization of Hamiltonian isotopies~\cite{GKS12}]}\label{th:gks12}
Let $\Phi\cl \dTM\times I\to \dTM$ be a homogeneous Hamiltonian isotopy.
Then there exists $K\in\Derlb(\cor_{M\times M\times I})$ satisfying
\banum
\item
$\musupp(K)\subset\Lambda_\Phi\cup T^*_{M\times M\times I}(M\times M\times I)$,
\item
$K_0\simeq \cor_\Delta$.
\eanum
Moreover:
\bnum
\item
both projections $\supp(K)\tto M\times I$ are proper,
\item
setting $\opb{K_s}\eqdot\opb{v}\rhom(K_s,\omega_M\etens\cor_M)$ with $v\cl M\times M\to M\times M$, $v(x,y)=(y,x)$, we have 
$K_s\conv\opb{K_s}\simeq \opb{K_s}\conv K_s\simeq \cor_\Delta$
for all $s\in I$,
\item
such a $K$ satisfying the conditions
(a) and (b) above is unique up to a unique isomorphism.
\enum
\end{theorem}
\begin{example}
Let $M=\R^n$ and denote by $(x;\xi)$ the homogeneous symplectic
coordinates on $T^*\R^n$. Consider the isotopy 
$\phi_s(x;\xi)= (x-s\frac{\xi}{\vert\xi\vert};\xi)$, $s\in I=\R$.
One proves that there exists $K\in\Derb(\cor_{M\times M\times I})$ such that 
$K_s$ denoting its restriction to $M\times M\times\{s\}$, one has: 
$K_s\simeq \cor_{\{\vert x-y\vert\leq s\}}$ for $s\geq0$
and  $K_s\simeq\cor_{\{\vert x-y\vert<-s\}}[n]$ for $s<0$.
\end{example}

\begin{theorem}[{\cite{GKS12}}]\label{th:arnold1}
Consider  a homogeneous Hamiltonian isotopy
$\Phi=\{\phi_s\}_{s\in I}$ $\cl \dTM\times I\to \dTM$ and 
 a $C^1$-map $\psi\cl M \to \R$
such that the differential $d\psi(x)$ never vanishes. Set
\eqn
&& \Lambda_\psi \eqdot \{(x;d\psi(x)); \; x\in M\}\subset \dTM.
\eneqn
Let $F_0\in\Derb(\cor_M)$ with compact support and such that $\rsect(M;F_0)\not=0$.
Then for any $s\in I$, $\phi_s(\musupp(F)\cap\dTM)\cap\Lambda_\psi\neq\varnothing$.
\end{theorem}
\begin{proof}[Sketch of proof]
Set $F_s=K_s\circ F_0$. Then $F_s$ has compact support and $\musupp(F_s)\cap\dTM=\phi_s(\musupp(F_0))\cap\dTM$. 
The direct image of $K\circ F_0$ on $I$ is a constant sheaf and this implies the isomorphism
$\rsect(M;F_s)\simeq\rsect(M;F_0)$. Then the result follows from Corollary~\ref{cor:Morse1}.
\end{proof}
It is possible to deduce (with some work, see loc.\ cit.\ Th.~4.16)  Arnold's conjecture from this result by choosing $M=N\times\R$ and $K=\cor_{N\times\{0\}}$.

\subsubsection*{Legendrian knots}

We assume now that $\cor$ is a field 
and  we  consider a closed smooth conic Lagrangian submanifold $\Lambda$ of $\dT^*M$. 
Recall Definition~\ref{def:simpleshv}.

It follows from Theorem~\ref{th:gks12} that 
the category $\Simple(\Lambda,\cor)$ is a Hamiltonian  isotopy invariant. (Note that this category may be empty. Conditions which ensure that it is not empty are obtained in~\cite{Gu12}.)

\begin{corollary}\label{cor:gks12}
Let $\Phi$ be a homogeneous Hamiltonian isotopy as in {\rm Theorem~\ref{th:gks12}}. Let $\Lambda_0$ be a smooth closed  conic Lagrangian 
submanifold of $\dT^*M$ and let $\Lambda_1=\phi_1(\Lambda_0)$. The categories 
$\Simple(\Lambda_0;\cor)$ and $\Simple(\Lambda_1;\cor)$ are equivalent. 
\end{corollary}

\begin{example}
In $\R^2$ with coordinates $(x,y)$ we define the following locally closed subset with boundary the cusp
\begin{equation}\label{eq:interior_cusp}
W = \{(x,y); \; x>0, \; -x^{3/2} \leq y < x^{3/2}\} .
\end{equation}
Outside the zero section,
$\musupp(\cor_W)$ is the smooth Lagrangian submanifold
\begin{equation}\label{eq:Lambda_cusp}
\Lambda_{cusp} = \{(t^2,t^3;-3 tu,2u); \;t\in\R, \; u>0\}.    
\end{equation}
\end{example}

Now assume that $M=N\times\R$ with $\dim N=1$, denote by $(t;\tau)$ the coordinates on $T^*\R$. Assume that $\Lambda$ is a closed conic smooth Lagrangian submanifold contained in  the open set $\{\tau>0\}$ of $T^*M$ and that its projection on $M$ is compact. 
By considering the image of $\Lambda$ in $T^*N\times\R$ by the map $(x,t;\xi,\tau)\mapsto (x,t;\xi/\tau)$ one gets a closed Legendrian smooth submanifold of the contact manifold $T^*N\times\R$. Its image by the projection $\pi$ is called a Legendrian knot, or a link, in $M$. Generically a link is a curve in $M$ with ordinary double points and cusps as its only singularities. A natural and important problem is  to find invariants which allow one to distinguish different links. The category $\Simple(\Lambda,\cor)$ is such an invariant and in~\cite{STZ14} the authors show that it allows to distinguish the two so-called Chekanov knots, a result which was not obtained with the traditional methods. 
\subsection{Conclusion}
We have discussed here a few applications of microlocal analysis or microlocal sheaf theory, 
first to systems of linear partial differential equations, next to symplectic topology. 

However, there are other applications, in particular in representation theory (see~\cite{HTT08}), in  singularity theory (see~\cite{Di04}) and, since   
quite recently,  in algebraic geometry. Indeed, 
Alexander Beilinson has constructed
the microsupport of constructible sheaves on schemes (see~\cite{Be15}) and this  new theory is being developed by several authors (see in particular~\cite{Sa15}). 


\providecommand{\bysame}{\leavevmode\hbox to3em{\hrulefill}\thinspace}

\vspace*{1cm}
\noindent
\parbox[t]{21em}
{\scriptsize{
\noindent
Pierre Schapira\\
Sorbonne Universit{\'e}s, UPMC Univ Paris 6\\
Institut de Math{\'e}matiques de Jussieu\\
e-mail: pierre.schapira@imj-prg.fr\\
http://webusers.imj-prg.fr/\textasciitilde pierre.schapira/
}}


\begin{bibdiv}
\begin{biblist}

\bib{An07} {article}{
author={Andronikof, Emmanuel},
title={Interview with Mikio Sato},
journal={Notices of the AMS},
volume={54, 2}, 
year={2007}
}

 \bib{Be15}  {article}{
author= {Beilinson, Alexander},
title={Constructible sheaves are holonomic},
eprint={arXiv:1505.06768v7},
year={2015}
}


\bib{BM88}{article}{
author={Bierstone, E},
author={Milman, P.~D.},
title={Semi-analytic and subanalytic sets},
journal={Publ. Math. IHES},
volume= {67}, 
date={1988},
pages={5-42}
}

\bib{BK89}{article}{
author={Bondal, Alexei},
author={Kapranov, Mikhail},
title={Representable functors, Serre functors, and mutations},
journal={Izv. Akad. Nauk SSSR},
volume={53}
year={1989}
pages={1183-1205}
}

\bib{BI73} {article}{
author={Bros, Jacques},
author={Iagolnitzer, Daniel},
title={Causality and local analyticity: mathematical study},
journal={Annales Inst. Fourier},
volume={18},
pages={147--184},
year={1973}
}
 

\bib{Di04}{book}{
author={Dimca, Alexandru},
title={Sheaves in topology},
series={Universitext},
publisher={Springer},
year={2004}
}
\bib{EM45} {article}{
author= {Eilenberg, S},
author= { Mac~Lane,  S},
title ={General theory of natural equivalences}, 
journal={Trans. Amer. Math. Soc. 58},
year={1945},
pages ={231--294} 
}

\bib{Ga81} {article}{
author={Gabber, Ofer},
title={The integrability of the characteristic variety},
journal={Amer. Journ. Math.},
volume= {103},
year={1981},
pages={ 445--468}
}

 \bib{Gri08}{article}{
 author={Ghrist, Robert},
title={Barcodes: The persistent topology of data},
journal={Bull. Amer. Math. Soc. },
volume={45},
year={2008},
pages={61-75}
}

\bib{Gr57}  {article}{
author= {Grothendieck,  Alexander},
title={Sur quelques points d'alg{\`e}bre homologique} , 
journal={Tohoku Math. Journ} ,
year={1957} ,
 pages ={119-183}
 }
\bib{Gu12}{article}{
 author={Guillermou, St\'ephane},
title={Quantization of conic Lagrangian submanifolds of cotangent bundles},
eprint={arXiv:1212.5818},
date={2012}
}

\bib{Gu13}{article}{
 author={Guillermou, St\'ephane},
title={The Gromov-Eliashberg theorem by microlocal sheaf theory},
eprint={arXiv:1311.0187},
date={2013}
}

\bib{Gu16}{article}{
 author={Guillermou, St\'ephane},
title={The three cusps conjecture},
eprint={arXiv:1603.07876},
date={2016}
}

\bib{GKS12}{article} {
 author={Guillermou, St\'ephane},
  author={Kashiwara, Masaki},
 author={Schapira, Pierre},
 title={Sheaf quantization of Hamiltonian isotopies and applications to non-displaceability problems}, 
 journal={Duke Math Journal},
 volume= {161},
 date={2012},
  pages={201Ð245}
}

\bib{Ho71} {article}{
author={H{\"o}rmander, Lars},
title={Fourier integral operators. I. },
 volume={127},
journal={Acta Math.},
page={79Ð183}
 year={1971}
 }

\bib{Ho83} {book}{
author={H{\"o}rmander, Lars},
title={The analysis of linear partial differential operators},
 series={Grundlehren der Mathematischen Wissenschaften},
 volume={256, 257 and 274, 275},
 publisher={Springer-Verlag, Berlin},
 year={1983, 1985}
 }
%
\bib{HTT08}{book}{
author={Hotta, R},
author={Takeuchi, K},
author={Tanisaki, T},
title={D-Modules, Perverse Sheaves, and Representation Theory},
series={Progress in Math.},
volume={236},
publisher= {Birkhauser},
year={2008},
pages={x+ 407}
}

\bib{JS16}{article}{
author={Jubin, Beno\^it},
author={Schapira, Pierre},
title={Sheaves and D-modules on causal manifolds,},
journal={Letters in Mathematical Physics},
volume={16},
date={2016},
pages={607-648}
}


\bib{Ka70}{book}{
author={Kashiwara, Masaki},
title={Algebraic study of systems of partial differential equations },
 note={Translated from author's thesis in Japanese, Tokyo 1970, by A. D'Agnolo and J-P.Schneiders},
publisher= {Soc. Math. France},
series={M{\'e}moires SMF},
volume={63},
year={1995},
 pages={xiii+72}
}

\bib{Ka03}{book}{
   author={Kashiwara, Masaki},
   title={$D$-modules and microlocal calculus},
   series={Translations of Mathematical Monographs},
   volume={217},
   publisher={American Mathematical Society, Providence, RI},
   date={2003},
   pages={xvi+254},
}

\bib{KS82} {article}{
author={Kashiwara, Masaki},
author={Schapira, Pierre},
title={Microsupport des faisceaux: applications aux modules diff{\'e}rentiels},
journal={C.~R.~Acad.\ Sci.\ Paris},
volume={295, 8},
pages={487--490},
date={1982}
}

\bib{KS85} {book}{
author={Kashiwara, Masaki},
author={Schapira, Pierre},
title={Microlocal study of sheaves},
 series={Ast{\'e}risque},
 volume={128},
 publisher={Soc.\ Math.\ France},
 date={1985}
 }

\bib{KS90}{book}{
 author={Kashiwara, Masaki},
 author={Schapira, Pierre},
 title={Sheaves on manifolds},
 series={Grundlehren der Mathematischen Wissenschaften},
 volume={292},
 publisher={Springer-Verlag, Berlin},
 date={1990},
 pages={x+512},
}

\bib{KS97}{article}{
author={Kashiwara, Masaki},
author={Schapira, Pierre},
title={Integral transforms with exponential kernels and Laplace transform},
journal={ Journal of the AMS},
volume={10},
pages={939-972},
date={1997}
}

\bib{KS06}{book}{
   author={Kashiwara, Masaki},
   author={Schapira, Pierre},
   title={Categories and sheaves},
   book={  series={Grundlehren der Mathematischen Wissenschaften},
   volume={332},
   publisher={Springer-Verlag, Berlin}},
   date={2006},
   pages={x+497},
}

\bib{Mr67} {article}{
author={Martineau, Andr\'e},
title={Th\'eor\`emes sur le prolongement analytique du type ``Edge of the Wedge''},
journal={Sem. Bourbaki},
volume={340},
year={1967/68}
}

\bib{Na09} {article}{
author={Nadler, David},
title={Microlocal branes are constructible sheaves},
journal={Selecta Math.},
volume={15},  
pages={563--619},
year={2009}
}

\bib{NZ09}{article}{
author={Nadler, David},
author={Zaslow, Eric},
title={Constructible sheaves and the Fukaya category},
journal={J. Amer. Math. Soc.},
volume= {22}, 
pages={233--286}, 
year={2009}
}

\bib{Sa15}{article}{
 author={Saito, Takeshi},
title={The characteristic cycle and the singular support of a constructible sheaf},
eprint={arXiv:1510.03018},
date={2015}
}

\bib{Sa60} {article}{
author={Sato, Mikio},
title={Theory of hyperfunctions,  I \& II},
journal={Journ. Fac. Sci. Univ. Tokyo} ,
volume={8},
pages={139--193, 487--436},
year={1959, 1960}
}

\bib{Sa70} {article}{
author={Sato, Mikio},
title={Regularity of hyperfunctions solutions of partial differential equations},
publisher={Actes du Congrs International des MathŽmaticiens, Gauthier-Villars, Paris},
volume={2},
pages={785--794},
year={1970}
}

\bib{SKK73}{article}{
 author={Sato, Mikio},
 author={Kawai, Takahiro},
 author={Kashiwara, Masaki},
 title={Microfunctions and pseudo-differential equations},
 conference={
 title={Hyperfunctions and pseudo-differential equations (Proc. Conf., Katata, 1971; dedicated to the memory of Andr\'e Martineau)},
 },
 book={publisher={Springer, Berlin}},
 date={1973},
 pages={265--529. Lecture Notes in Math., Vol. 287}
}

\bib{Sc07} {article}{
author={Schapira, Pierre},
 title={Mikio Sato, a visionary of mathematics},
journal={Notices of the AMS},
volume ={54, 2},
date={2007}
}

%

\bib{ScSn94}{book}{
author={Schapira, Pierre},
author={Schneiders, Jean-Pierre},
title={Index theorem for elliptic pairs},
 series={Ast{\'e}risque},
 volume={224}, 
 publisher={Soc.\ Math.\ France},
 date={1994}
 }
 
\bib{SGA4} {article}{
author={SGA4},
author={Artin, Mike},
author={Grothendieck, Alexander},
author={Verdier, Jean-Louis},
title={Th{\'e}orie des topos et cohomologie {\'e}tale des sch{\'e}mas}, 
book={
title={S{\'e}m. G{\'e}om. Alg. (1963--64) },
 series={Lecture Notes in Math.},
  publisher={Springer, Berlin},
 volume={269, 270, 305},
year={1972/73}}
}
 

\bib{STZ14}{article} {
author={Shende, Vivek},
author={Treumann, David},
author={Zaslow, Eric},
title={Legendrian knots and constructible sheaves},
eprint={arXiv:1402.0490},
date={2014}
}


\bib{Sj82} {book}{
author={Sj\"ostrand, Johannes},
title={Singularit\'es analytiques microlocales},
series={Ast{\'e}risque},
volume={95},
publisher={Soc.\ Math.\ France},
year={1982}
}


\bib{Ta08}{article} {
author={Tamarkin, Dmitry},
title={Microlocal conditions for non-displaceability},
eprint={arXiv:0809.1584},
date={2008}
}

\end{biblist}
\end{bibdiv}
\end{document}